\documentclass[11pt]{amsart}

\usepackage{amsthm}
\usepackage{amssymb}
\usepackage{amsfonts}
\usepackage{thmtools}
\usepackage{booktabs}
\usepackage{fullpage}
\usepackage{color}
\usepackage{enumitem}
\usepackage{tikz}
\usepackage[english]{babel}
\usepackage{microtype}

\definecolor{nicegreen}{RGB}{0,158,115}
\definecolor{nicered}{RGB}{213,94,0}
\definecolor{niceblue}{RGB}{0,114,178}
\definecolor{niceorange}{RGB}{230,159,0}
\definecolor{nicepink}{RGB}{204,121,167}
\definecolor{nicesky}{RGB}{86,180,233}
\definecolor{niceyellow}{RGB}{230,218,56}

\usetikzlibrary{decorations.pathmorphing,decorations.pathreplacing,calligraphy}

\usepackage[
defernumbers=true,
maxbibnames=6,
backend=biber,
style=alphabetic,
giveninits=true,
sortcites=false,
maxalphanames=4,
minalphanames=3,
doi=true,
isbn=false,
url=false,
safeinputenc,
natbib,
sorting=nyt,
]{biblatex}
\usepackage[colorlinks,allcolors=niceblue]{hyperref}
\usepackage[capitalise,english,noabbrev]{cleveref}

\newtheorem{thm}{Theorem}
\newtheorem{lmm}[thm]{Lemma}
\newtheorem{crl}[thm]{Corollary}
\newtheorem{prp}[thm]{Proposition}
\newtheorem{fct}[thm]{Fact}

\theoremstyle{definition}
\newtheorem{dfn}[thm]{Definition}

\newtheorem{rmk}[thm]{Remark}
\newtheorem{qst}[thm]{Question}
\crefname{thm}{Theorem}{Theorems}

\newcommand{\ts}{\textstyle}

\let\phi\varphi
\let\epsilon\varepsilon
\let\emptyset\varnothing
\let\subset\subseteq
\let\supset\supseteq
\let\nsubset\nsubseteq

\newcommand{\cal}{\mathcal}
\newcommand{\bb}{\mathbb} 
\renewcommand{\bf}{\mathbf}
\renewcommand{\sf}{\mathsf}
\renewcommand{\rm}{\mathrm}
\newcommand{\fr}{\mathfrak}

\newcommand{\emp}{\emptyset}
\renewcommand{\Cap}{\bigcap}
\renewcommand{\Cup}{\bigcup}

\newcommand{\pto}{\rightharpoonup}
\newcommand{\md}{\vDash} 
\newcommand{\fc}{\Vdash} 
\newcommand{\imp}{\Rightarrow}
\newcommand{\pmi}{\Leftarrow}

\newcommand{\leftrightdelim}[4][]{\def\leftrightdelimparamunused{#1}\ifx\leftrightdelimparamunused\empty\mathopen{}\left#2#4\right#3\mathclose{}\else\mathopen{}\big#2{#4}\big#3\mathclose{}\fi}
\newcommand{\leftrightseparator}[1][]{\def\leftrightseparatorparamunused{#1}\ifx\leftrightseparatorparamunused\empty~\middle|~\else~\big|~\fi}

\newcommand{\ac}[2][]{\leftrightdelim[#1]{\{}{\}}{#2}}
\newcommand{\sqb}[2][]{\leftrightdelim[#1]{[}{]}{#2}}
\newcommand{\card}[2][]{\leftrightdelim[#1]{|}{|}{#2}}
\newcommand{\ab}[2][]{\leftrightdelim[#1]{\langle}{\rangle}{#2}}

\newcommand{\set}[3][]{\ac[#1]{#2\leftrightseparator[#1]#3}}
\newcommand{\abset}[3][]{\ab[#1]{#2\leftrightseparator[#1]#3}}

\newcommand{\iv}[3][]{\leftrightdelim[#1]{[}{)}{#2,#3}}
\newcommand{\ap}[1]{\text{\normalfont``}#1\text{\normalfont''}} 

\newcommand{\dom}{\mathrm{dom}}
\newcommand{\ran}{\mathrm{ran}}
\newcommand{\ot}{\mathrm{ot}}
\newcommand{\cf}{\mathrm{cf}}
\newcommand{\cov}{\mathrm{cov}}
\newcommand{\non}{\mathrm{non}}
\newcommand{\cof}{\mathrm{cof}}
\newcommand{\add}{\mathrm{add}}
\newcommand{\Ord}{\bf{Ord}}

\newcommand{\omcs}{{}^{\omega}2}
\newcommand{\omom}{{}^{\omega}\omega}
\newcommand{\mucs}{{}^{\mu}2}
\newcommand{\kacs}{{}^{\kappa}2}
\newcommand{\mumu}{{}^{\mu}\mu}
\newcommand{\kamu}{{}^{\kappa}\mu}
\newcommand{\muka}{{}^{\mu}\kappa}
\newcommand{\kaka}{{}^{\kappa}\kappa}
\newcommand{\lamu}{{}^{\lambda}\mu}
\newcommand{\laka}{{}^{\lambda}\kappa}
\newcommand{\lala}{{}^{\lambda}\lambda}
\newcommand{\lanu}{{}^{\lambda}\nu}
\newcommand{\lacs}{{}^{\lambda}2}
\newcommand{\nuka}{{}^{\nu}\kappa}
\newcommand{\nucs}{{}^{\nu}2}
\newcommand{\derh}{{}^{\delta}\rho}

\newcommand{\fmumu}{{}^{<\mu}\mu}
\newcommand{\flanu}{{}^{<\lambda}\nu}
\newcommand{\fkamu}{{}^{<\kappa}\mu}
\newcommand{\fderh}{{}^{<\delta}\rho}

\newcommand{\spshh}[3]{{#1}_{(#2,#3)}}

\newcommand{\Mkamuka}{\spshh{\cal M}{\kamu}{\kappa}}
\newcommand{\Mmumubd}{\spshh{\cal M}{\mumu}{\rm{bd}}}
\newcommand{\Mmumumu}{\spshh{\cal M}{\mumu}{\mu}}
\newcommand{\Mmucsmu}{\spshh{\cal M}{\mucs}{\mu}}
\newcommand{\Mmumuka}{\spshh{\cal M}{\mumu}{\kappa}}
\newcommand{\Mmukaka}{\spshh{\cal M}{\muka}{\kappa}}
\newcommand{\Mmucska}{\spshh{\cal M}{\mucs}{\kappa}}
\newcommand{\Mkacska}{\spshh{\cal M}{\kacs}{\kappa}}
\newcommand{\Mkakaka}{\spshh{\cal M}{\kaka}{\kappa}}
\newcommand{\MXtau}{\spshh{\cal M}{X}{\tau}}
\newcommand{\MYsigma}{\spshh{\cal M}{Y}{\sigma}}
\newcommand{\MXbd}{\spshh{\cal M}{X}{\rm{bd}}}

\newcommand{\Cmukaka}{\spshh{\bb C}{\muka}{\kappa}}
\newcommand{\Cmucska}{\spshh{\bb C}{\mucs}{\kappa}}
\newcommand{\Ckacska}{\spshh{\bb C}{\kacs}{\kappa}}
\newcommand{\Cnucsnu}{\spshh{\bb C}{\nucs}{\nu}}
\newcommand{\CXtau}{\spshh{\bb C}{X}{\tau}}

\newcommand{\dmukabd}{\spshh{\fr d}{\muka}{\rm{bd}}}
\newcommand{\dkakabd}{\spshh{\fr d}{\kaka}{\rm{bd}}}
\newcommand{\dmukamu}{\spshh{\fr d}{\muka}{\mu}}
\newcommand{\dkakaka}{\spshh{\fr d}{\kaka}{\kappa}}
\newcommand{\dmukaka}{\spshh{\fr d}{\muka}{\kappa}}
\newcommand{\dkakaal}{\spshh{\fr d}{\kaka}{\rm{all}}}
\newcommand{\dmukaal}{\spshh{\fr d}{\muka}{\rm{all}}}
\newcommand{\dlakaal}{\spshh{\fr d}{\laka}{\rm{all}}}
\newcommand{\dnukaal}{\spshh{\fr d}{\nuka}{\rm{all}}}
\newcommand{\dlakala}{\spshh{\fr d}{\laka}{\lambda}}

\newcommand{\bmukabd}{\spshh{\fr b}{\muka}{\rm{bd}}}
\newcommand{\bkakabd}{\spshh{\fr b}{\kaka}{\rm{bd}}}
\newcommand{\bmukamu}{\spshh{\fr b}{\muka}{\mu}}
\newcommand{\bmukaka}{\spshh{\fr b}{\muka}{\kappa}}
\newcommand{\bkakaka}{\spshh{\fr b}{\kaka}{\kappa}}
\newcommand{\bmukaal}{\spshh{\fr b}{\muka}{\rm{all}}}
\newcommand{\bkakaal}{\spshh{\fr b}{\kaka}{\rm{all}}}

\newcommand{\relplaceholder}{\bullet}
\newcommand{\dXph}{\spshh{\fr d}{X}{\relplaceholder}}
\newcommand{\bXph}{\spshh{\fr b}{X}{\relplaceholder}}
\newcommand{\dXka}{\spshh{\fr d}{X}{\kappa}}
\newcommand{\bXka}{\spshh{\fr b}{X}{\kappa}}
\newcommand{\dXmu}{\spshh{\fr d}{X}{\mu}}
\newcommand{\bXmu}{\spshh{\fr b}{X}{\mu}}
\newcommand{\dXbd}{\spshh{\fr d}{X}{\rm{bd}}}
\newcommand{\bXbd}{\spshh{\fr b}{X}{\rm{bd}}}
\newcommand{\dXal}{\spshh{\fr d}{X}{\rm{all}}}
\newcommand{\bXal}{\spshh{\fr b}{X}{\rm{all}}}
\newcommand{\bXnu}{\spshh{\fr b}{X}{\nu}}

\newcommand{\code}{\sf{code}}
\newcommand{\Code}{\sf{Code}}
\newcommand{\decode}{\sf{decode}}
\newcommand{\Decode}{\sf{Decode}}
\newcommand{\dcset}{\sf{set}}
\newcommand{\Set}{\sf{Set}}

\newcommand{\restr}{\mathord{\restriction}}

\newcommand{\append}{{}^\frown}
\newcommand{\altappend}{\kern1pt{}^\land\kern1pt}

\title{Topology and Category for Singular Product Spaces}
\author{Yusuke Hayashi}
\address{
	Graduate School of System Informatics,
	Kobe University,
	1-1, Rokkodai-cho, Nada-ku, 657-8501, Kobe, 
	Japan
}
\email{219x504x@stu.kobe-u.ac.jp}
\author{Tristan van der Vlugt}
\address{%
	Institut für Diskrete Mathematik und Geometrie, 
	Technische Universität Wien,
	Wiedner Hauptstrasse 8-10/104,
	1040 Wien,
	Austria
}
\email{tristan@tvdvlugt.nl}

\thanks{The second author was funded by the Austrian Science Fund (FWF) [grants PAT8484324 and P33895] during the writing of this article.}

\begin{document}

	\begin{abstract}
		For $\kappa$ a regular uncountable cardinal, the higher Baire and Cantor spaces $\kaka$ and $\kacs$ (endowed with the ${<}\kappa$-box topology) have been relatively well-studied, but less is known about the case where $\kappa$ is singular. We will consider several spaces of functions and box topologies that could serve as higher Baire and Cantor spaces for singular cardinals. The ultimate focus of the article lies in studying cardinal characteristics of the ideal of $\kappa$-meagre subsets of these spaces.
	\end{abstract}
	
	\maketitle
	
	\section{Introduction}
	
	The study of topological and combinatorial properties of the continuum forms one of the classical cornerstones of both set theory and general topology. For combinatorial and set theoretic purposes, it is often simpler to work with zero-dimensional spaces; rather than working with the real line directly, it is a common practice to work with the Cantor space ($\omcs$, or the set of functions from the natural numbers $\omega$ to the two-element set $2=\ac{0,1}$) and the Baire space ($\omom$, of functions on the natural numbers), to such extent that both of these spaces are frequently referred to as ``the reals''. 
	
	One may generalise these two product spaces by replacing the natural numbers $\omega$ with an uncountable cardinal $\kappa$, for instance to obtain the space $\kacs$ with the ${<}\kappa$-box topology (which is the least $\kappa$-additive topology extending the product topology), known as the \emph{higher Cantor space}, or \emph{generalised Cantor space}. The higher Cantor space and related spaces have seen a significant amount of interest in recent years, but the subject goes back the better part of a century. In 1950, \citet{Sikorski50} laid the topological foundation for $\kappa$-additive spaces of weight $\kappa$, among which the higher Cantor space $\kacs$ under the assumption that $\kappa=\kappa^{<\kappa}$. Many familiar concepts of the classical reals, such as Borel sets, metrics and Baire category, were generalised to regular uncountable $\kappa$ by Sikorski, producing the higher analogues of $\kappa$-Borel sets, $\omega_\mu$-metrics and $\kappa$-Baire category. Moreover, Sikorski remarks the surprising fact that the higher Cantor space $\kacs$ and higher Baire space $\kaka$ are homeomorphic when $\kappa$ is not strongly inaccessible; this stands in contrast against our classical intuition, since $\omcs$ is compact, whereas $\omom$ is not. Subsequently, \citet{HungNegrepontis73} improved Sikorski's result and showed that, for regular uncountable $\kappa$, $\kacs$ and $\kaka$ are homeomorphic if and only if $\kappa$ is not a weakly compact cardinal (a large cardinal notion that is strictly stronger than inaccessibility).
	
	In many cases, the knowledge and intuition we have about set theory of the classical Baire space transfers to the higher context. In this article we will be particularly interested in generalising the classical Cicho\'n diagram (see \cite[Chapter 2]{BartoszynskiJudah95}). This diagram consists of ten cardinal characteristics of the continuum, that is, cardinalities that lie between $\aleph_1$ and the size of the continuum $2^{\aleph_0}$. We will not define these at this point, but we will show the diagram and its higher generalisation:

	\begin{center}
		\begin{tikzpicture}[xscale=1.9, yscale=1]
			\node (cN) at (1,2) {$\rm{cov}(\cal N)$};
			\node (nM) at (2,2) {$\rm{non}(\cal M)$};
			\node (b) at (2,1) {$\fr b$};
			\node (d) at (3,1) {$\fr d$};
			\node (cM) at (3,0) {$\rm{cov}(\cal M)$};
			\node (nN) at (4,0) {$\rm{non}(\cal N)$};
			\node (aM) at (2,0) {$\rm{add}(\cal M)$};
			\node (fM) at (3,2) {$\rm{cof}(\cal M)$};
			\node (aN) at (1,0) {$\rm{add}(\cal N)$};
			\node (fN) at (4,2) {$\rm{cof}(\cal N)$};
			
			\draw (aN) edge[->] (cN);
			\draw (cN) edge[->] (nM);
			\draw (b) edge[->] (nM);
			\draw (b) edge[->] (d);
			\draw (cM) edge[->] (d);
			\draw (cM) edge[->] (nN);
			\draw (nN) edge[->] (fN);
			\draw (aN) edge[->] (aM);
			\draw (aM) edge[->] (b);
			\draw (aM) edge[->] (cM);
			\draw (nM) edge[->] (fM);
			\draw (d) edge[->] (fM);
			\draw (fM) edge[->] (fN);
			
			\node (hnM) at (6,2) {$\rm{non}(\cal M_\kappa)$};
			\node (hb) at (6,1) {$\fr b_\kappa$};
			\node (hd) at (7.25,1) {$\fr d_\kappa$};
			\node (hcM) at (7.25,0) {$\rm{cov}(\cal M_\kappa)$};
			\node (haM) at (6,0) {$\rm{add}(\cal M_\kappa)$};
			\node (hfM) at (7.25,2) {$\rm{cof}(\cal M_\kappa)$};
			
			\draw (hb) edge[->] (hnM);
			\draw (hb) edge[->] (hd);
			\draw (hcM) edge[->] (hd);
			\draw (haM) edge[->] (hb);
			\draw (haM) edge[->] (hcM);
			\draw (hnM) edge[->] (hfM);
			\draw (hd) edge[->] (hfM);

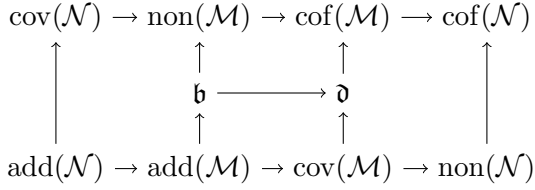
			\node at (2.5,-0.8) {Classical Cicho\'n diagram};
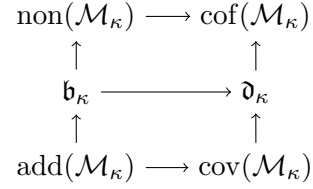
			\node at (6.625,-.8) {Higher Cicho\'n diagram for $\kappa=\kappa^{<\kappa}$};
			
		\end{tikzpicture}
	\end{center}
	
	The Cicho\'n diagram provides an ordering between cardinal characteristics, where an arrow $\fr x\to\fr y$ implies that $\mathsf{ZFC}$ proves the inequality $\fr x\leq\fr y$. The diagram is complete, because a strict inequality between any two of the cardinal characteristics that does not contradict the diagram is consistent relative to the consistency of $\mathsf{ZFC}$.
	
	\citet{CummingsShelah95} gave an early example of higher analogues of $\fr b$ and $\fr d$ to the higher Baire space $\kaka$ for regular uncountable $\kappa$. More recently, \citet{BrendleBrookeTaylorFriedmanMontoya18} gave a comprehensive overview of the cardinal invariants of the $\kappa$-meagre ideal $\cal M_\kappa$ for a regular uncountable cardinal $\kappa$ satisfying $\kappa=\kappa^{<\kappa}$, which is drawn above. \citet{Brendle22} also studied the Cicho\'n diagram under the assumption that $\kappa$ is regular uncountable with $\kappa<2^{<\kappa}$.

	Much less studied is what happens if we relax the requirement that $\kappa$ is regular, and work with the higher Baire space $\mumu$ where we assume $\mu$ is a singular cardinal. In the literature, the topology and combinatorial set theory of singular higher Baire spaces have been sporadically addressed under very specific circumstances. Early work by \citet{Stone62} established that, for singular cardinals $\mu$ with countable cofinality, the product space ${}^\omega\mu$ is homeomorphic to the product space $\prod_{n\in\omega}\mu_n$, where $\abset{\mu_n}{n\in\omega}$ is an increasing sequence of cardinals cofinal in~$\mu$. 
	The study of the cardinals of the Cicho\'n diagram to singular higher product spaces has also seen some coverage. \citet{Miller82} and \citet{Landver92,Landver93} studied the covering number of the nowhere dense ideal (and thus of Baire category) for higher product spaces $\kacs$, including for singular $\kappa$ (cf.\ \cref{thm:LandverMiller}), and the dominating number $\fr d_\mu$ on the space $\mumu$ for singular $\mu$ has been studied by \citet{Shelah19}, \citet{GartiShelah20}, \citet{LambieHanson23} and \citet{Hayashi26}.
	
	One particularly active area of research concerned with the higher Baire space for singular cardinals, is generalised descriptive set theory. Although descriptive set theory has been studied for regular cardinals, e.g.\ in \cite{MeklerVaananen93,FriedmanHyttinenKulikov14,AgostiniMottoRosSchlicht23}, the situation is especially interesting for singular cardinals with countable cofinality: the product space ${}^\omega\mu$ that was introduced by \citeauthor{Stone62} has a particularly rich $\mu$-Borel structure and is metrisable. This space has been studied intensely in the last decade by \citet{DimonteMottoRos25} and their co-authors. Although our article does not treat definability and regularity properties, we would like to mention the thematically related study of the Baire property in ${}^\omega\mu$ and its relation to rank-into-rank cardinals, cf.\ \cite{DimonteIannellaLucke25,DimontePovedaThei24}.
	
	Although for descriptive set theoretic purposes, the space ${}^{\omega}\mu$ for singular cardinals $\mu$ with countable cofinality undeniably seems to produce the most interesting results (see \cite{DimonteMottoRos25}), it is \emph{a~priori} not obvious whether other choices of product space or topology can be interesting from the perspective of cardinal characteristics, or indeed of the Cicho\'n diagram. We show that other spaces than ${}^\omega\mu$ are indeed interesting. We will define a collection of twelve potential spaces, each of which could be claimed to be a generalisation of the classical Cantor or Baire space. Our first contribution to the topic is topological in nature, as we will provide a complete overview of the conditions under which these twelve spaces are homeomorphic to one another. Our second contribution is to study the cardinal characteristics of the (higher) meagre ideal associated to each of these spaces. Indeed, the choice of space is significant, and we will show that it leads to distinct cardinal characteristics.

	\subsection{Overview of this paper}\label{sec:Structure} 
	
	In \cref{sec:Topology} we will consider several spaces that consist of functions on ordinals equipped with different generalisations of the product topology. We will give a full overview of the conditions under which these spaces are homeomorphic to each other. In order to do this, we compute some topological invariants of these spaces, such as their weight and character. The remainder of this article will study certain cardinal characteristics on each of the spaces defined in \cref{sec:Topology}. 
	
	In \cref{sec:Covering} we recall the notion of meagreness, we generalise it to the higher context and treat its additivity and covering numbers. In most of our spaces the additivity and covering numbers equal $\kappa^+$ (that is, the values are absolute). In \cref{sec:Uniformity} we look at the uniformity numbers. We show that these are highly dependent on the choice of space. We give several consistency results using the method of forcing to prove that several of the uniformity numbers are independent of each other and of their respective lower and upper bounds. In \cref{sec:Dominating} we briefly review the dominating and unbounding numbers for singular cardinals. Finally, in \cref{sec:Cofinality} we discuss a lower bound for the cofinality numbers in some of our spaces, modifying a result by \citet{Brendle22} to give a connection between the cofinality numbers and certain dominating numbers. As a consequence we obtain the consistency of $2^\mu$ being strictly smaller than some of the cofinality numbers.

	\subsection{Preliminaries}\label{sec:Preliminaries}
	Throughout this article, we will always assume that $\kappa$ is a regular cardinal and that $\mu$ is a singular cardinal with $\cf(\mu)=\kappa$. Note bene: this assumption will not be stated in theorems or lemmas.
	
	Most of our notation agrees with the notation in standard references of set theory, such as \cite{Jech03,Kunen11}. We will use greek letters both to denote cardinal numbers and ordinal numbers, where $\kappa,\lambda,\mu,\nu,\theta$ depict cardinal numbers (and as mentioned $\kappa$ and $\mu$ have fixed roles), whereas $\alpha,\beta,\gamma,\delta,\eta,\zeta,\xi$ usually depict ordinal numbers. As is usual in set-theoretic writing, we may interpret cardinal numbers as ordinal numbers, and we do not make a semantical distinction between $\alpha<\beta$ and $\alpha\in\beta$. There exists some ambiguity between cardinal and ordinal arithmetic, but we assume the reader understands which type is used based on context and what makes sense.

	For two cardinal numbers $\lambda,\nu$, we distinguish between $\lanu$ and $\nu^\lambda$, the former being the set of functions from $\lambda$ to $\nu$ and the latter being the cardinality of such a set of functions. We also write $\flanu$ for the set of functions $\Cup_{\alpha\in\lambda}{}^\alpha\nu$. A function $f$ is seen as a family of pairs containing exactly one pair $(x,y)$ for each $x$ in the domain $\dom(f)$. Given $f\colon X\to Y$ and $A\subset X$, we write $f\restr A$ for the restriction of $f$ to the domain $A$. Such functions $f\restr A$ are also called \emph{partial functions} from $X$ to $Y$, and we write $f\colon X\pto Y$ if $f$ is a partial function (that is, if there is $A\subset X$ such that $f\colon A\to Y$). By  $\rm{Fn}_{\lambda}(X,Y)$ we denote the family of partial functions from $X$ to $Y$ with a domain of size ${<}\lambda$. If $s\in{}^\alpha\nu$ and $t\in{}^\beta\nu$ for some ordinals $\alpha,\beta$, then $s\append t\in{}^{\alpha+\beta}\nu$ denotes the concatenation of $s$ with $t$.
	We write $[\nu]^\lambda$ for the family of subsets of $\nu$ of size $\lambda$, and similarly for $[\nu]^{<\lambda}$.
	
	At several points in this article, we will need the following notion from cardinal arithmetic.
	\begin{dfn}\label{dfn:cf covering numbers}
		For $\theta\leq\nu$ infinite cardinalities, we define $\cf(\sqb\theta^{<\nu})$ as the least size of a $\subset$-cofinal subset of $\sqb\theta^{<\nu}$, that is, the least size of a family $\cal F\subset\sqb\theta^{<\nu}$ such that for every $X\in\sqb\theta^{<\nu}$ there exists some $Y\in\cal F$ with $X\subset Y$. We write $\cf(\sqb\theta^\nu)$ for $\cf(\sqb\theta^{<\nu^+})$.
	\end{dfn}
	These cardinal characteristics are special cases of \emph{covering numbers} from cardinal arithmetic, and in the notation of \cite{Shelah94} we have $\cf(\sqb\theta^{<\nu})=\cov(\theta,\nu,\nu,2)$. We refer to \cite[\S\,5]{Shelah94} for an extensive overview of these cardinal characteristics, but remark the following basic facts.
	\begin{fct}Let $\mu$ be singular and $\cf(\mu)=\kappa$.
		\begin{enumerate}[label=(\arabic*)]
			\item\label{lem:cf mu kappa facts:1} $\mu\leq\cf(\sqb\mu^{<\kappa})\leq \mu^{<\kappa}$ and $\mu^+\leq\cf(\sqb\mu^{\kappa})\leq \mu^\kappa$ and $\mu^+\leq\cf(\sqb\mu^{<\mu})\leq \mu^{<\mu}$.
			\item $\mu^\kappa=2^\kappa\cdot\cf(\sqb\mu^\kappa)$.
			\item $\cf(\sqb\kappa^{<\kappa})=\kappa$.
		\end{enumerate}
	\end{fct}
	
	For several results in \cref{sec:Covering,sec:Uniformity}, we assume the reader is familiar with the basics of forcing, e.g., such as treated in \cite{Jech03} or \cite{Kunen11}. We force downwards, thus $q\leq p$ implies that $q$ is a stronger condition than $p$.

	\section{Topological considerations}\label{sec:Topology}

	Before we can study the higher Baire and Cantor spaces for singular cardinals, we have to properly define what these spaces are. When we generalise the classical context of $\omega$ to the higher context of a singular cardinal $\mu$, we first have to consider what property of $\omega$ we are trying to generalise: should we replace $\omega$ by $\mu$ or by its cofinality $\kappa$? The choice depends on whether we want to interpret $\omega$ as a cardinality or as a cofinality. Exhausting all such choices in the generalisation of $\omom$ and $\omcs$, we end up with four sets of functions $\mumu$, $\muka$, $\mucs$ and $\kamu$ as potential underlying sets for the higher Baire and Cantor spaces. 
	
	The ubiquitous choice for a topology on $\omom$ (and $\omcs$) is the product topology. Remember that a subbase for the product topology is given by cylinder sets of the form $C_{n,m}=\set{f\in\omom}{f(n)=m}$. To obtain a base, we close this subbase under finite intersections. The topologies we will consider in the higher context will have bases given by intersections of  cylinder sets as well, but we have to choose how to generalise the notion of a \emph{finite} intersection. The concept of {``finiteness''}---the natural choice of smallness with respect to $\omega$---can be understood in three ways: in terms of cardinality (finite cardinalities are strictly smaller than $\omega$), cofinality (finite sets in $\omega$ are bounded),  or a combination of both (the cardinality of a finite set is too small to be cofinal). For a regular cardinal, these three interpretations coincide, but for a singular cardinal they do not, and hence we will consider three topologies.

	\subsection{Definition of the topological spaces}\label{sec:Topology definition}
	
	Let us, for the purpose of giving the definition of the topologies, fix our set of functions as the set $\derh$. That is, $\delta$ is the domain and $\rho$ is the range. Let $s\colon \delta\pto \rho$ be a partial functions from $\delta$ to $\rho$, then we define the \emph{cone} of $s$ (relative to $\derh$) as:
	\[
	\sqb s=\set{f\in\derh}{s\subset f}.
	\]
	Clearly, each cone $[s]$ is an intersection of cylinder sets in $\derh$; precisely, $[s]=\Cap_{\alpha\in\dom(s)}C_{\alpha,s(\alpha)}$. Each of the topologies we consider has a base consisting of cones of a specific family of partial functions.

	The \emph{${<}\mu$-box topology} on $\derh$ is generated by sets that are the intersection of ${<}\mu$-many cylinder sets. Equivalently, a base for the ${<}\mu$-box topology is given by $\set{\sqb s}{s\in\rm{Fn}_\mu(\delta,\rho)}$. In much the same way, we define the \emph{${<}\kappa$-box topology} as being generated by intersections of ${<}\kappa$-many cylinder sets; equivalently, the base is given by $\set{[s]}{s\in\rm{Fn}_\kappa(\delta,\rho)}$. Finally, the \emph{bounded topology} is generated by cones $[s]$ for $s\colon \delta\pto\rho$ partial functions with a domain bounded in $\delta$; equivalently, the family of basic open sets is given by $\set{\sqb s}{s\in\fderh}$.
	
	Assuming that $\cf(\delta)=\kappa$, we can easily observe that these three topologies are directly comparable to each other, and that the ${<}\kappa$-box topology is the coarsest of the three, whereas the ${<}\mu$-box topology is the finest of the three. Moreover, $\delta<\mu$ implies that the ${<}\mu$-box topology is discrete, and $\delta=\kappa$ (which is regular) implies that the ${<}\kappa$-box and bounded topologies coincide. We will write $(\derh,\kappa)$, $(\derh,\rm{bd})$ and $(\derh,\mu)$ for the space $\derh$ equipped with, respectively, the ${<}\kappa$-box, bounded, and ${<}\mu$-box topologies. Similarly, we will abbreviate these topologies by ``$\kappa$'', ``$\rm{bd}$'' and ``$\mu$'' in various places, e.g., in subscripts.
	
	To summarise, we have four sets of functions---$\mumu$, $\muka$, $\mucs$ and $\kamu$---and three topologies---bounded, ${<}\kappa$-box and ${<}\mu$-box---thus a total of twelve spaces to consider, each space $(X,\tau)$ consisting of one of the sets of functions $X$ and one of the topologies $\tau$. It is easy to see that each of these spaces is a zero-dimensional completely regular Hausdorff space. Moreover, apart from $(\kamu,\mu)$ (which is discrete), each of these spaces is strongly homogeneous (i.e., every nonempty clopen set is homeomorphic to the whole space).\footnote{This follows for instance from theorems by \citet[Theorem~2.4]{Terada93} and \citet[Corollary~14]{Medini11}.}

	Some of these twelve spaces happen to be homeomorphic to each other, whereas others are not homeomorphic to each other. In many cases, however, the existence of a homeomorphism is independent of $\mathsf{ZFC}$. The remainder of this section will be used to provide a complete description of the conditions under which the existence of a homeomorphism between any pair of our twelve spaces is provable. A schematic overview of these results is given in \cref{fig:homeomorphisms}. The reader who is purely interested in cardinal characteristics of the meagre ideal may safely skip forward to \cref{rmk:spaces}.

	\begin{figure}[t]
		
		\begin{tikzpicture}[xscale = 2.5, yscale=1.5]
			\node (kamubd) at (0,2) {$(\kamu,\rm{bd})$};
			\node (kamuka) at (0,1) {$(\kamu,\kappa)$};
			\node (kamumu) at (0,0) {$(\kamu,\mu)$};
			
			\node (mucsbd) at (1,2) {$(\mucs,\rm{bd})$};
			\node (mucska) at (1,1) {$(\mucs,\kappa)$};
			\node (mucsmu) at (1,0) {$(\mucs,\mu)$};
			
			\node (mukabd) at (2,2) {$(\muka,\rm{bd})$};
			\node (mukaka) at (2,1) {$(\muka,\kappa)$};
			\node (mukamu) at (2,0) {$(\muka,\mu)$};
			
			\node (mumubd) at (3,2) {$(\mumu,\rm{bd})$};
			\node (mumuka) at (3,1) {$(\mumu,\kappa)$};
			\node (mumumu) at (3,0) {$(\mumu,\mu)$};

			\draw[line width=1.5pt, shorten <=1mm, shorten >=1mm] (kamubd) -- (kamuka);
			
			\draw[line width=1.5pt, shorten <=1mm, shorten >=1mm] (mucsbd) -- (mukabd);
			
			\draw[line width=1.5pt, shorten <=1mm, shorten >=1mm] (mucsmu) -- (mukamu);
			
			\draw[] (kamubd) edge[line width=1.5pt, dash pattern=on 2.2mm off 1mm, edge label={{\footnotesize (A)}}, shorten <=1mm, shorten >=1mm] (mucsbd);
			\draw[] (mukabd) edge[line width=1.5pt, dash pattern=on 2.2mm off 1mm, edge label={\footnotesize (A)}, shorten <=1mm, shorten >=1mm] (mumubd);
			
			\draw[] (mukaka) edge[line width=1.5pt, dash pattern=on 2.2mm off 1mm, edge label={\footnotesize (C)}, shorten <=1mm, shorten >=1mm] (mumuka);
			
			\draw[] (mucska) edge[line width=1.5pt, dash pattern=on 2.2mm off 1mm, edge label={{\footnotesize (B)}}, shorten <=1mm, shorten >=1mm] (mukaka);
			
			\draw[] (mukamu) edge[line width=1.5pt, dash pattern=on 2.2mm off 1mm, edge label={\footnotesize (D)}, shorten <=1mm, shorten >=1mm] (mumumu);
			
			\draw (3.6,0) -- (3.6,2.3);
			
			\node [right] at (3.65,2.15) {\scshape Legend.};
			\node [right] at (3.65,1.55) {Homeomorphism exists iff ...};
			\node [right] at (3.65,1.2) {(A)\quad spaces have equal weight};
			\node [right] at (3.65,.85) {(B)\quad $\kappa$ is not weakly compact\footnotemark{}};
			\node [right] at (3.65,.5) {(C)\quad $\mu\leq 2^{<\kappa}$};
			\node [right] at (3.65,.15) {(D)\quad $\mu$ is not a strong limit};
		\end{tikzpicture}
		\caption{Overview of spaces}
		\label{fig:homeomorphisms}
	\end{figure}
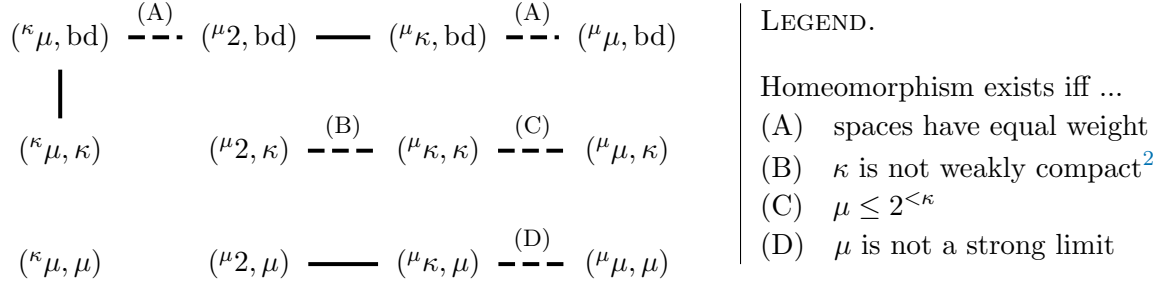

	\subsection{Topological invariants}

	It is often possible to show that two spaces cannot be homeomorphic by showing that the spaces have distinct topological invariants. The topological invariants that serve most useful to us, are the weight, character and pseudocharacter of the spaces. We recall their definitions. Let $(X,\tau)$ be a topological space. 
	
	\begin{dfn}
		The \emph{weight} $w(X,\tau)$ is the least cardinality of a base for $\tau$. The \emph{character} $\chi(X,\tau)$ is the least cardinality $\nu$ such that each point $x\in X$ has a local base of cardinality $\nu$ around $x$. The \emph{pseudocharacter} $\psi(X,\tau)$ is the least cardinality $\nu$ such that each singleton $\ac x\subset X$ is the intersection of a family of open sets of size $\nu$. 
	\end{dfn}
	A summary of the values of the weight, character and pseudocharacter for each of the twelve spaces is given in \cref{fig:invariants}.
	
	\footnotetext{Under our definition, we regard $\omega$ as a weakly compact cardinal, cf.\ \cref{dfn:weakly compact}.}
	
	\begin{table}[t]
		
		\caption[Overview of topological invariants]{Overview of topological invariants\footnotemark{}}\label{fig:invariants}
		
		\begin{tabular}{rccccccccc}
			\toprule
			&
			\multicolumn{3}{c}{bounded} & 
			\multicolumn{3}{c}{${<}\kappa$-box} &
			\multicolumn{3}{c}{${<}\mu$-box} 
			\\\cmidrule(lr){2-4}\cmidrule(lr){5-7}\cmidrule(lr){8-10}
			& 
			$\kamu$ & $\mucs$ & $\mumu$ &
			$\mucs$ & $\muka$ & $\mumu$ &
			$\kamu$ & $\mucs$ & $\mumu$ 
			\\
			\midrule
			$w(X,\tau)=$ & 
			$\mu^{<\kappa}$ & $2^{<\mu}$ & $\mu^{<\mu}$ &
			$\mu^{<\kappa}$ 
			& $\mu^{<\kappa}$ 
			& $\mu^{<\kappa}$ &
			$\mu^\kappa$ & 
			$\mu^{<\mu}$  
			& $\mu^{<\mu}$
			\\
			$\chi(X,\tau)=$ & 
			$\kappa$ & $\kappa$ & $\kappa$ &
			$\cf(\sqb\mu^{<\kappa})$ & $\cf(\sqb\mu^{<\kappa})$ & $\cf(\sqb\mu^{<\kappa})$ & 
			$1$ & $\rm{cf}(\sqb\mu^{<\mu})$ & $\rm{cf}(\sqb\mu^{<\mu})$
			\\
			$\psi(X,\tau)=$ &
			$\kappa$ & $\kappa$ & $\kappa$ &
			$\mu$ & $\mu$ & $\mu$ &
			$1$ & $\kappa$ & $\kappa$ 
			\\
			\bottomrule
		\end{tabular}
		
	\end{table}
	\footnotetext{The character and pseudocharacter of each space is given in \cref{thm:top invariants}, and the weights for the bounded spaces are computed in \cref{lem:weight sized open partition}. The other spaces are left to the reader: for $(\mucs,\kappa)$ and $(\muka,\kappa)$, note that $\mu^{<\kappa}=2^{<\kappa}\cdot \cf([\mu]^{<\kappa})$, and for $(\mucs,\mu)$ similarly note that $\mu^{<\mu}=2^{<\mu}\cdot\cf(\sqb\mu^{<\mu})$}
	
	\begin{prp}\label{thm:top invariants}
		\leavevmode
		\begin{enumerate}[label=(\arabic*)]
			\item \label{thm:top invariants:1} If $X\in\ac{\kamu,\mucs,\muka,\mumu}$, then $\chi(X,\rm{bd})=\psi(X,\rm{bd})=\kappa$. Also, $\chi(\kamu,\kappa)=\psi(\kamu,\kappa)=\kappa$.
			\item \label{thm:top invariants:2}
			If $X\in\ac{\mucs,\muka,\mumu}$, then $\chi(X,\kappa)=\cf(\sqb\mu^{<\kappa})$ and $\psi(X,\kappa)=\mu$.
			\item \label{thm:top invariants:3} If $X\in\ac{\mucs,\muka,\mumu}$, then $\chi(X,\mu)=\cf(\sqb\mu^{<\mu})$ and $\psi(X,\mu)=\kappa$.
			\item \label{thm:top invariants:4} $(\kamu,\mu)$ is discrete, thus $\chi(\kamu,\mu)=\psi(\kamu,\mu)=1$.
		\end{enumerate}
	\end{prp}
	\begin{proof}
		The calculations are easy. For example, to compute the character, note that, without loss of generality, any local base around a point $x\in\derh$ consists of cones of the form $\sqb{x\restr A}$, where $A\subset \delta$. If we take $(\mucs,\kappa)$ as our space, then $\chi(\mucs,\kappa)=\cf(\sqb\mu^{<\kappa})$ holds because $F\subset\sqb\mu^{<\kappa}$ is cofinal if and only if $\set{\sqb{x\restr A}}{A\in F}$ is a local base.
	\end{proof}
	The points \ref{thm:top invariants:1}--\ref{thm:top invariants:4} categorise our twelve spaces into four kinds, and any two spaces of different kind cannot be homeomorphic, as they have incompatible topological invariants.
	
	\subsection{The bounded topology spaces}
	For the spaces with the bounded topology, we will show that any of them is consistently homeomorphic to any other, and this consistency depends on whether the spaces have equal weight. We also repeat that the ${<}\kappa$-box and bounded topologies coincide for the set of functions $\kamu$, and thus $(\kamu,\kappa)=(\kamu,\rm{bd})$. 
	
	The existence of homeomorphisms between the spaces with bounded topologies can be shown by constructing a base $\cal B$ in each space such that the partial orders $(\cal B,\subset)$ of the two bases are isomorphic. We will use the following notion: 
	
	\begin{dfn}
		let $(X,\tau)$ be a space, then an \emph{open partition} of $X$ is a family $P$ of mutually disjoint nonempty open sets such that $\Cup P=X$.
	\end{dfn}
	
	\begin{lmm}\label{lem:weight sized open partition}
		We have the following weights: 
		\[
		w(\kamu,\rm{bd})=\mu^{<\kappa},\quad 
		w(\mucs,\rm{bd})=w(\muka,\rm{bd})=2^{<\mu},\quad
		w(\mumu,\rm{bd})=\mu^{<\mu}.
		\]
		Furthermore, $(X,\rm{bd})$ has an open partition of cardinality $w(X,\rm{bd})$,  for each $X\in\ac{\kamu,\mucs,\muka,\mumu}$.
	\end{lmm}
	\begin{proof}
		Note that the size of an open partition is a lower bound to the weight. For $(\kamu,\rm{bd})$, an open partition is given by $\set{\sqb s}{s\in\fkamu\text{ and }\dom(s)\neq\emp\text{ and } s(0)=\dom(s)}$. For $(\mucs,\rm{bd})$, fix some $f\colon \kacs\to\mu$ such that $\ran(f)$ is cofinal in $\mu$, then $\set{[s\append t]}{s\in\kacs\text{ and } t\in {}^{f(s)}2}$ is an open partition. Open partitions for $(\muka,\rm{bd})$ and $(\mumu,\rm{bd})$ are found in the same way.
	\end{proof}

	We remark that for the spaces $(\mucs,\rm{bd})$ and $(\mumu,\rm{bd})$, the specific conditions under which they are homeomorphic have already been described in \cite[Proposition~6.6]{AndrettaMottoRos22}.\footnote{The characterisation given in \cite{AndrettaMottoRos22} is that $(\mucs,\rm{bd})$ and $(\mumu,\rm{bd})$ are homeomorphic if and only if $\mu$ is not a singular strong limit. This is the same condition that follows from \cref{thm:homeo for bounded}: $2^{<\mu}=\mu^{<\mu}$ implies $\mu$ is not a strong limit, since $\mu<\mu^{<\mu}$; reversely, if $\mu$ is not a strong limit and $\lambda<\mu$ is such that $2^\lambda\geq\mu$, then $2^\lambda=\mu^\lambda$.  }
	
	\begin{thm}\label{thm:homeo for bounded}
		If $X,X'\in\ac{\kamu,\mucs,\muka,\mumu}$, then $(X,\rm{bd})$ is homeomorphic to $(X',\rm{bd})$ if and only if $w(X,\rm{bd})=w(X',\rm{bd})$.
	\end{thm}
	\begin{proof}
		Let $X,X'\in\ac{\kamu,\mucs,\muka,\mumu}$ and assume $w(X,\rm{bd})=w(X',\rm{bd})$. We build a sequence of open partitions $\abset{B_\alpha}{\alpha\in\kappa}$ of $(X,\rm{bd})$ with the following two properties:
		\begin{enumerate}
			\item For $\alpha<\beta$ and $[s]\in B_\alpha$, the set $\set{[t]\in B_\beta}{[t]\subset[s]}$ has size $w(X,\rm{bd})$. Note that this set is an open partition of $\sqb s$, it exists by \cref{lem:weight sized open partition} and because $[s]$ and $X$ are homeomorphic by strong homogeneity of $(X,\rm{bd})$.
			\item For $f\in X$ and $\abset{s_\alpha}{\alpha\in\kappa}$ such that $[s_\alpha]\in B_\alpha$ and $f\in [s_\alpha]$ for all $\alpha$, we have $f=\Cup_{\alpha\in\kappa} s_\alpha$.
		\end{enumerate}
		Observe that $\cal B:=\Cup_{\alpha\in\kappa} B_\alpha$ forms a base for $(X,\rm{bd})$. We also build a base $\cal B'$ for $(X',\rm{bd})$ following the same method. Since $w(X,\rm{bd})=w(X',\rm{bd})$, it is easy to see that the structures $(\cal B,\subset)$ and $(\cal B',\subset)$ are isomorphic. Thus $(X,\rm{bd})$ and $(X',\rm{bd})$ are homeomorphic.
	\end{proof}
	It follows that $(\mucs,\rm{bd})$ and $(\muka,\rm{bd})$ are homeomorphic under $\sf{ZFC}$, and that $\mu^{<\kappa}=2^{<\mu}$ and $2^{<\mu}=\mu^{<\mu}$, respectively, are equivalent to the existence of homeomorphisms between $(\kamu,\rm{bd})$ and $(\mucs,\rm{bd})$, and between $(\mucs,\rm{bd})$ and $(\mumu,\rm{bd})$. 
	
	Let us discuss the consistency of these arithmetical assertions. Both are independent of $\sf{ZFC}$: 
	\begin{enumerate}
		\item Under the assumption of $\sf{GCH}$ we have $\mu^{<\kappa}=2^{<\mu}=\mu<\mu^+=\mu^{<\mu}$. 
		\item Adding $\mu^{+}$-many Cohen generic subsets of $\kappa$ over a model of $\sf{GCH}$, we obtain a model in which $\mu^{<\kappa}=\mu<2^\kappa=\mu^+=2^{<\mu}=\mu^{<\mu}$. 
		\item If $\kappa$ is uncountable, then both equalities may hold in the same model, 
		by adding $\mu^+$-many Cohen generic subsets of $\omega$ to a model of $\sf{GCH}$, which forces $2^{\aleph_0}=\mu^+=\mu^{<\kappa}=2^{<\mu}=\mu^{<\mu}$. 
	\end{enumerate}
	Cardinal arithmetic allows us to conclude one additional result under $\sf{ZFC}$ without additional assumptions:
	
	\begin{crl}
		The space $(\mucs,\rm{bd})$ is homeomorphic to $(\kamu,\rm{bd})$ or to $(\mumu,\rm{bd})$ or to both.
	\end{crl}
	\begin{proof}
		First note that the weights of the four bounded spaces are linearly ordered: 
		\[
		\mu^{<\kappa}\leq 2^{<\mu}=\kappa^{<\mu}\leq \mu^{<\mu}.
		\]
		If we assume $\mu^{<\kappa}<\kappa^{<\mu}$, then there is $\lambda<\mu$ such that $\mu^{<\kappa}<\kappa^\lambda$. But, $\kappa<\mu\leq\kappa^\lambda$ implies $\mu^\lambda=\kappa^\lambda$ (e.g., by \cite[Theorem~5.20]{Jech03}), and thus it follows that $\kappa^{<\mu}=\mu^{<\mu}$.
	\end{proof}
	
	\subsection{\texorpdfstring{The ${<}\kappa$-box topology spaces}{The <κ-box topology spaces}}
	Let us now consider the spaces with the ${<}\kappa$-box topology. As far as the regular higher Cantor and Baire spaces are concerned, it is well-known that $(\kaka,\kappa)$ and $(\kacs,\kappa)$ are homeomorphic if and only if  $\kappa$ is not a weakly compact cardinal (cf. \cite[Theorem~1]{HungNegrepontis73}). There are many equivalent ways to define weakly compact cardinals (see \cite[Theorem~8.23]{ComfortNegrepontis74}), but the one we will use is topological in nature.
	\begin{dfn}\label{dfn:weakly compact}
		Let a topological space $X$ be called \emph{${<}\kappa$-compact} if every open cover of $X$ has a subcover of cardinality ${<}\kappa$. We define an infinite cardinal number $\kappa$ to be \emph{weakly compact} if $(\kacs,\kappa)$ is a ${<}\kappa$-compact space. 
	\end{dfn}
	Although unconventional, for the sake of brevity we allow weakly compact cardinals to be countable, and indeed $\aleph_0$ satisfies the above conditions, since $(\omcs,\omega)$ is a compact space. In our context of spaces of functions with domain $\mu$, we will show that $(\mucs,\kappa)$ and $(\muka,\kappa)$ are homeomorphic under the same assumption that the cofinality $\kappa$ of $\mu$ is weakly compact. In the proof of \cref{thm:weakly compact homeomorphism} below, the direction $(\pmi)$ was told to us by Hiroshi Sakai, who granted permission to include it in this article.
	
	\begin{thm}\label{thm:weakly compact homeomorphism}
		The spaces $(\mucs,\kappa)$ and $(\muka,\kappa)$ are homeomorphic if and only if $\kappa$ is not a weakly compact cardinal.
	\end{thm}
	\begin{proof}
		$(\imp)$ 
		If $\kappa$ is not weakly compact, then $(\kaka,\kappa)$ and $(\kacs,\kappa)$ are homeomorphic. Since $(\muka,\kappa)$ is the $\mu$-length product of $(\kaka,\kappa)$, and $(\mucs,\kappa)$ is the $\mu$-length  product of $(\kacs,\kappa)$, both with the ${<}\kappa$-box topology, it follows that $(\mucs,\kappa)$ and $(\muka,\kappa)$ are homeomorphic.
		
		$(\pmi)$ Suppose that $\kappa$ is weakly compact. We claim that the space $(\mucs,\kappa)$ has no open partition of size $\kappa$. Clearly $(\muka,\kappa)$ does have such open partition, making the spaces not homeomorphic. 
		
		Let $\set{X_\alpha}{\alpha\in\kappa}$ be an open cover of $(\mucs,\kappa)$  and, for each $\alpha$, choose $S_\alpha\subset \rm{Fn}_\kappa(\mu,2)$ of minimal cardinality such that $X_\alpha=\Cup_{s\in S_\alpha}[s]$. It suffices to find $\alpha,\beta\in\kappa$ such that $X_\alpha\cap X_\beta\neq\emp$. There are two cases to consider: 
		\begin{enumerate}[label=(\arabic*)]
			\item\label{thm:weakly compact homeomorphism:case1} $\card{S_\alpha}\leq\kappa$ for all $\alpha$, 
			\item\label{thm:weakly compact homeomorphism:case2} $\card{S_\alpha}>\kappa$ for some $\alpha$. 
		\end{enumerate}
		
		\textit{Case \ref{thm:weakly compact homeomorphism:case1}.}  Let $B:=\Cup_{s\in\Cup_{\alpha\in\kappa}S_\alpha}\dom(s)$,  then $\card{B}\leq \kappa$.  Working in the space $({}^{B}2,\kappa)$ we define $Y_\alpha=\Cup_{s\in S_\alpha}[s]$,  then $\set{Y_\alpha}{\alpha\in\kappa}$ is an open cover of ${}^{B}2$.  Since $\kappa$ is weakly compact and $({}^B2,\kappa)\cong(\kacs,\kappa)$, we find that ${}^{B}2$ is ${<}\kappa$-compact,  so there is $\beta<\kappa$ such that $\set{Y_\alpha}{\alpha\in\beta}$ covers ${}^{B}2$.  Hence $Y_\alpha\cap Y_\beta\neq\emp$ for some $\alpha<\beta$,  and thus $X_\alpha\cap X_\beta\neq\emp$.
		
		\textit{Case \ref{thm:weakly compact homeomorphism:case2}.} Fix some $\alpha$ such that $\card{S_\alpha}>\kappa$ and recursively construct a sequence $\abset{s_\xi}{\xi\in\kappa^+}$ of elements of $S_\alpha$ as follows. Given 
		\[
		A:=\set[big]{s_\eta\restr D}{\eta<\xi\text{\,, }D\subset\dom(s_\eta)\text{ and }\sqb{s_\eta\restr D}\subset X_\alpha},
		\]
		we have $\card{A}\leq\kappa$, thus $[A]:=\Cup_{s\in A}[s]$ is a proper subset of  $X_\alpha$. Then we define $s_\xi$ such that $s_\eta\restr D\in A$ implies $s_\xi\restr D\neq s_\eta\restr D$ for all $\eta<\xi$ and $D\in\sqb\mu^{<\kappa}$. By the $\Delta$-system lemma (see, e.g., \cite[Lemma III.6.15]{Kunen11}), we may find a root $R\in[\mu]^{<\kappa}$, a set of indices $\Xi\subset \kappa^+$ of size $\card{\Xi}=\kappa^+$ and $s^*\colon R\to 2$ such that $\dom(s_\xi)\cap \dom(s_{\xi'})=R$ and $s^*\subset s_\xi$ for all distinct $\xi,\xi'\in \Xi$. 
		Now note that the recursive construction implies $\sqb{s^*}\nsubset X_\alpha$,  therefore we can choose $x\in\sqb{s^*}\setminus X_\alpha$. Then $x\in X_\beta$ for some $\beta\neq \alpha$.  Let $t\in S_\beta$ be such that $x\in\sqb t$,  then there is $\xi\in \Xi$ such that $\dom(s_\xi\setminus s^*)\cap \dom(t)=\emp$ by the fact that the sets $\dom(s_\xi\setminus s^*)$ are mutually disjoint for different $\xi\in\Xi$.  Finally, $\dom(t)\cap \dom(s_\xi)\subset\dom(s^*)$,  so $\sqb{s_\xi}\cap\sqb{t}\neq\emp$.  Therefore we see that $X_\alpha\cap X_\beta\neq\emp$.
	\end{proof}

	\begin{thm}\label{thm:homeo mumu ka and muka ka}
		The spaces $(\muka,\kappa)$  and $(\mumu,\kappa)$ are homeomorphic if and only if $\mu\leq2^{<\kappa}$.
	\end{thm}
	\begin{proof}
		$(\pmi)$ Assume that $\mu\leq2^{<\kappa}$, then we show that $(\kaka,\kappa)$ and $(\kamu,\kappa)$ are homeomorphic. Namely, if $\mu\leq2^{<\kappa}$, then both $(\kaka,\kappa)$ and $(\kamu,\kappa)$ have an open partition of size $2^{<\kappa}$, therefore we can apply the method from the proof of \cref{thm:homeo for bounded} to show that $(\kamu,\kappa)$ and $(\kaka,\kappa)$ are homeomorphic. This makes both $(\muka,\kappa)$ and $(\mumu,\kappa)$ equal to $\mu$-length products of the same space, and thus shows that $(\muka,\kappa)$ and $(\mumu,\kappa)$ are homeomorphic as well.

		$(\imp)$ Assume that $2^{<\kappa}<\mu$, then a standard application of the $\Delta$-system lemma shows that $(\muka,\kappa)$ has no family of $(2^{<\kappa})^+$-many mutually disjoint open sets, whereas $(\mumu,\kappa)$ clearly does.\footnote{The $\Delta$-system lemma can be applied, since $\kappa\leq\lambda^+$ are regular cardinals and, by regularity of $\kappa$, $(2^{<\kappa})^{<\kappa}=2^{<\kappa}$.}
	\end{proof}
	
	We make some notes on the consistency of homeomorphisms between the spaces with the ${<}\kappa$-box topology. As uncountable weakly compact cardinals are a large cardinal notion, $\sf{ZFC}$ alone cannot prove the existence of an uncountable weakly compact cardinal. If $\kappa$ is uncountable, then $\mu\leq2^{<\kappa}$ is independent of $\sf{ZFC}$: it is false under $\sf{GCH}$, but holds in the model obtained by adding $\mu^+$-many Cohen generic subsets of $\omega$.  If $\kappa=\omega$ is countable, then $2^{<\kappa}=\omega$, and therefore $2^{<\kappa}<\mu$ holds under $\sf{ZFC}$. Finally we note that $\mu\leq2^{<\kappa}$ is incompatible with $\kappa$ being weakly compact, since weakly compact cardinals are strong limit cardinals, thus if $(\mumu,\kappa)$ is homeomorphic to $(\muka,\kappa)$, then it is also homeomorphic to $(\mucs,\kappa)$.

	\subsection{\texorpdfstring{The ${<}\mu$-box topology spaces}{The <μ-box topology spaces}}
	Finally we turn to the spaces with the ${<}\mu$-box topology.

	\begin{thm}\label{thm:homeo between mucs mu and muka mu}
		The spaces $(\mucs,\mu)$ and $(\muka,\mu)$ are homeomorphic.
	\end{thm}
	\begin{proof}
		Consider, in both spaces, the family of basic sets \[\set{\sqb s}{\dom(s)\text{ is a union of }\kappa\text{-intervals}},\] where we call a subset of $\mu$ a \emph{$\kappa$-interval} if it is of the form $\iv{\kappa\alpha}{\kappa\alpha+\kappa}$ for some $\alpha\in\mu$. It is easy to see that these two families form a base for their respective space, and since $2^\kappa=\kappa^\kappa$, there is an obvious homeomorphism.
	\end{proof}
	
	\begin{thm}\label{thm:homeo between mumu mu and muka mu}
		The spaces $(\mucs,\mu)$ and $(\mumu,\mu)$ are homeomorphic if and only if $\mu$ is not strong limit.
	\end{thm}
	\begin{proof}
		$(\imp)$ Suppose $\mu$ is strong limit, that is, $2^\lambda<\mu$ for all $\lambda<\mu$. We show that $(\mucs,\mu)$ does not have a family of mutually disjoint open sets of size $\mu^+$. This suffices, since in $(\mumu,\mu)$, the set $\set{\sqb s}{s\in\kamu}$ is an antichain of size $\mu^\kappa\geq\mu^+$. The $\Delta$-system lemma would be the natural choice to prove this, but does not apply, since $\mu$ is singular. We therefore prove the claim directly.
		
		Let $\cal A\subset\rm{Fn}_\mu(\mu,2)$ be a family of partial functions such that the basic open sets $\set{\sqb s}{s\in\cal A}$ are mutually disjoint. Without loss we may assume that there is some $\lambda<\mu$ such that $\card{s}=\lambda$ for all $s\in \cal A$. We fix a well-order $\prec$ on $\cal A$. Because none of the functions in $\cal A$ are compatible with each other, for any distinct $s,s'\in\cal A$ we have $D=\dom(s)\cap \dom(s')\neq\emp$ and $s\restr D\neq s'\restr D$.
		
		Construct a tree $T$ consisting of all sequences $\abset{(A_\alpha,s_\alpha,f_\alpha)}{\alpha\in\gamma}$, where $\gamma$ is an ordinal, with the following properties:
		\begin{enumerate}[label=(\roman*)]
			\item $\emp\neq A_\alpha\subset\cal A$ for each $\alpha\in\gamma$,
			\item $A_0=\cal A$,
			\item $s_\alpha$ is the $\prec$-minimal element of $A_\alpha$ for each $\alpha\in\gamma$
			\item\label{mukamu mumumu homeo: D set} $f_\alpha\in\rm{Fn}_\mu(\mu,2)$ is a function with $\dom(f_\alpha)=\dom(s_\alpha)\setminus\Cup_{\xi\in\alpha}\dom(s_\xi)$ and $f_\alpha\neq s_\alpha\restr \dom(f_\alpha)$ for each $\alpha\in\gamma$,
			\item $A_\alpha= \Cap_{\xi\in\alpha}A_\xi$ for each limit $\alpha\in\gamma$,
			\item $A_{\alpha+1}=\set{s\in A_\alpha}{f_\alpha\restr\dom(s)\subset s}$ for each successor $\alpha+1\in\gamma$.
		\end{enumerate}
		If $t=\abset{(A_\alpha,s_\alpha,f_\alpha)}{\alpha\in\gamma}\in T$, we will write $(A^t_\alpha,s^t_\alpha,f^t_\alpha)$ for the elements of the sequence $t$. We will also write $T_\gamma=\set{t\in T}{\ot(t)=\gamma}$. Let us make some observations about $T$.
		
		\begin{enumerate}[label=(\arabic*)]
			\item 
			If $t\in T_{\alpha+1}$ for some $\alpha$, then 
			\[
			\Cup\set{A_{\alpha+1}^{t'}}{t\subset t'\in T_{\alpha+2}}=A_\alpha^t\setminus\ac{s^t_\alpha},
			\]
			because for every $s\in A_\alpha^t\setminus\ac{s_\alpha^t}$ we have by construction that  $s\restr\Cup_{\xi\in\alpha}\dom(s_\xi^t)$ is compatible with $s_\alpha^t\restr\Cup_{\xi\in\alpha}\dom(s_\xi^t)$, and thus there is $\beta\in (\dom(s)\cap \dom(s_\alpha^t))\setminus \Cup_{\xi\in\alpha}\dom(s_\xi^t)$ for which $s(\beta)\neq s_\alpha^t(\beta)$. In particular, it follows from this that for every $s\in \cal A$ there is some $\alpha<\card{\cal A}^+$ and $t\in T_{\alpha+1}$ such that $s=s_\alpha^t$. Thus $\card{T}\geq\card{\cal A}$.
			\item We claim that $T$ has height $\leq \lambda^++1$. Otherwise, if $\abset{(A_\alpha,s_\alpha,f_\alpha)}{\alpha\leq\lambda^+}\in T_{\lambda^++1}$, then $\dom(s_{\lambda^+})\cap \dom(f_\alpha)\neq\emp$ for all $\alpha<\lambda^+$. But $\dom(f_\alpha)\cap \dom(f_\beta)=\emp$ for all $\alpha<\beta$, and thus we would have $\card{s_{\lambda^+}}=\lambda^+$, a contradiction. Therefore we have $T_{\lambda^++1}=\emp$.
			
			\item Note that $\card{\dom(f_\alpha)}\leq\lambda$. Moreover, if $t\in T_{\alpha}$ and $t',t''$ are distinct successors of $t$, then $A_{\alpha}^{t'}=A_{\alpha}^{t''}$, hence $s_\alpha^{t'}=s_\alpha^{t''}$ and thus we see that $f_\alpha^{t'}\neq f_\alpha^{t''}$. Therefore, each $t\in T$ has at most $2^\lambda$ successors. 
		\end{enumerate}
		Each node of $T$ splits into at most $2^\lambda$ successors and $T$ has height at most $\lambda^++1$, therefore
		\[
		\card{\cal A}\leq \card{T}\leq (2^\lambda)^{\lambda^+}=2^{\lambda^+}\leq 2^{<\mu}\leq\mu.
		\]

		$(\pmi)$ Let $\lambda<\mu$ be such that $2^\lambda\geq \mu$, then $2^\lambda=\mu^\lambda$, and thus $(\lacs,\mu)$ and $(\lamu,\mu)$ are homeomorphic (and discrete). This makes both $(\mucs,\mu)$ and $(\mumu,\mu)$ the $\mu$-length product with the ${<}\mu$-box topology of the same space.
	\end{proof}
	
	\begin{rmk}\label{rmk:spaces}
		Disregarding the discrete space $(\kamu,\mu)$, there are (up to homeomorphism) seven distinct spaces we will be considering. These are:
		\begin{enumerate}
			\item $\kamu$ and $\mumu$ with the bounded topology, 
			\item $\mucs$, $\muka$ and $\mumu$ with the ${<}\kappa$-box topology, and 
			\item $\mucs$ and $\mumu$ with the ${<}\mu$-box topology. 	
		\end{enumerate} 
		Remember that, although $(\mucs,\rm{bd})$ is not provably homeomorphic to any of these spaces in $\sf{ZFC}$, it is provably homeomorphic to one of $(\kamu,\rm{bd})$ and $(\mumu,\rm{bd})$, and thus we will not need to consider this space separately. 
	\end{rmk}

	\section{Category and its additivity and covering numbers}\label{sec:Covering}	
	
	Let us first review the classical notion of Baire category.
	Let $X$ be a topological space. A subset $N\subset X$ is \emph{nowhere dense} if every nonempty open $U$ has a nonempty open $V\subset U$ such that $N\cap V=\emp$, or equivalently if the interior of the closure of $N$ is empty. In particular, each nowhere dense is contained in a closed nowhere dense set, and the complement of a closed nowhere dense set is open dense. A set is called \emph{meagre} (or \emph{of first category}) if it is the countable union of nowhere dense sets, and otherwise it is called \emph{nonmeagre} (or \emph{of second category}). The complement of a meagre set is called \emph{comeagre}. The (classical) Baire category theorem states that, in any completely metrisable space, the countable intersection of open dense sets is dense. In other words, the Baire category theorem is equivalent to all comeagre sets being dense.
	
	When the Baire category theorem (and with it the notion of category) are generalised to the higher context, we may often prove a stronger theorem, in which countable intersections are replaced by intersections of a larger family of open dense sets. For a cardinal $\lambda$, let us call a set \emph{$\lambda$-meagre} if it is the union of $\lambda$-many nowhere dense sets. Exactly, if we consider the higher Baire space $\lala$ where $\lambda$ is regular, then the intersection of $\lambda$-many open dense subsets of $\lala$ is dense, or equivalently every $\lambda$-comeagre subset of $\lala$ is dense. We will see that in the higher context for a singular cardinal $\mu$, the cofinality $\cf(\mu)=\kappa$ plays the role of $\lambda$. We first establish that $\kappa$-meagre is the right notion to obtain something strictly more general than nowhere denseness. The proof of the following fact is elementary.
	
	\begin{fct}
		Let  $(X,\tau)$ be a space considered in \cref{rmk:spaces}. Then $\kappa$ is the least cardinal such that there exists a family of $\kappa$-many nowhere dense subsets of $X$ whose union is dense.\qed
	\end{fct}
	
	One way to describe the statement of the above fact, is that the nowhere dense ideal of $(X,\tau)$ is not \emph{$\kappa$-additive}. This leads us to the definition of four cardinal characteristics associated with an ideal $\cal I$ on a space $X$:
	
	\begin{dfn}\label{dfn:characteristics}
		Let $\cal I\subset \cal P(X)$ be a proper ideal on $X$, then we define the following cardinal characteristics.
		
		\begin{enumerate}
			\item The \emph{additivity number} $\add(\cal I)$ is the least size of a set $A\subset\cal I$ such that $\Cup A\notin\cal I$.
			\item The \emph{covering number} $\cov(\cal I)$ is the least size of a set $C\subset\cal I$ such that $\Cup C=X$,
			\item The \emph{uniformity number} $\non(\cal I)$ is the least size of a set $N\subset X$ with $N\notin\cal I$,
			\item The \emph{cofinality number} $\cof(\cal I)$ is the least size of a set $F\subset\cal I$ such that every $I\in \cal I$ has some $J\in F$ with $I\subset J$.
		\end{enumerate}
		It is well-known as well as easy to see that $\add(\cal I)\leq\cov(\cal I),\non(\cal I)$ and $\cov(\cal I),\non(\cal I)\leq \cof(\cal I)$.
	\end{dfn}
	
	For the space $(X,\tau)$, we define $\MXtau$ to be the family of $\kappa$-meagre sets. By definition it is clear that $\MXtau$ is a ${\leq}\kappa$-complete ideal, i.e., it is a family closed under subsets and unions of size $\kappa$. It is also easy to see that each of the spaces is not $\kappa$-meagre in itself, and thus the ideal $\MXtau$ is proper. In particular, the following theorem follows.
	\begin{thm}\label{additivity lower bound}
		Let  $(X,\tau)$ be a space considered in \cref{rmk:spaces}. Then $\add(\MXtau)\geq\kappa^+$.
	\end{thm}
	
	For most of the spaces from \cref{rmk:spaces}, the covering number of the $\kappa$-meagre ideal is at most, and hence equal to, $\kappa^+$. Our preferred method to prove this  concerns the forcing notion defined from the topology. 
	
	For a space $X$ with topology $\tau$, we define $\CXtau$ as the forcing notion consisting of open sets ordered by the subset relation. Since the basic open sets form a dense subset, we may assume that all conditions in $\CXtau$ are basic open sets. For each of the spaces from \cref{rmk:spaces}, the forcing notion $\CXtau$ is ${<}\kappa$-closed and thus preserves cofinalities ${\leq}\kappa$. We use the following folklore lemma.
	
	\begin{lmm}\label{lmm:collapse forcing lemma}
		If $\CXtau$ collapses $\kappa^+$, i.e., if there is a  $\CXtau$-name $\dot f$ with $\fc_{\CXtau}\ap{\dot f\colon(\kappa^+)^\bf V\to\kappa\text{ is injective}}$, then $\cov(\MXtau)\leq\kappa^+$.	
	\end{lmm}
	\begin{proof}
		Consider the sets  $D_{\alpha,\beta}=\set[b]{p\in\CXtau}{ p\fc\ap{\dot f(\alpha)=\beta}}$ for each $\alpha\in\kappa^+$ and $\beta\in\kappa$, and let $D_\alpha=\Cup_{\beta\in\kappa}D_{\alpha,\beta}$. The set $D_\alpha$ is an open dense subset of $\CXtau$ in the forcing sense, and thus $\Cup D_\alpha$ is open dense as a subset of $X$ in the topological sense. If $\alpha,\alpha'\in\kappa^+$ are distinct and $\beta\in\kappa$, then $\Cup D_{\alpha,\beta}\cap \Cup D_{\alpha',\beta}$ is empty (as otherwise it would force $\dot f$ to not be injective). Therefore by the pigeonhole principle $\Cap_{\alpha\in\kappa^+}\Cup D_\alpha=\emp$.
		Finally, note that $X\setminus\Cup D_\alpha$ is nowhere dense, and therefore $\Cup_{\alpha\in\kappa^+} (X\setminus \Cup D_\alpha)=X$.
	\end{proof}
	
	\begin{lmm}
		Let $(X,\tau)$ be a space considered in \cref{rmk:spaces}. If $(X,\tau)\neq(\mucs,\kappa)$ and $(X,\tau)\neq(\muka,\kappa)$, then $\CXtau$ collapses $\kappa^+$.
	\end{lmm}
	\begin{proof}
		These are standard results. See also \cite[Exercise VII.G5]{Kunen80} and \cite[Exercise 15.3]{Jech03}.
	\end{proof}
	
	This tells us that for most of the spaces under consideration, the additivity and covering numbers have the fixed value of $\kappa^+$. We remark that the specific case for $(\mucs,\mu)$ has been first observed by \citet[Lemma~1.3\,(c)]{Landver92}, indeed using \cref{lmm:collapse forcing lemma}.
	
	\begin{crl}\leavevmode
		
		\setlength{\tabcolsep}{2pt}
		\begin{center}
			\begin{tabular}{ccccccccccl}
				$\add(\Mkamuka)$& $=$ &$\add(\Mmumubd)$& $=$ &$\add(\Mmumuka)$& $=$ &$\add(\Mmucsmu)$& $=$ &$\add(\Mmumumu)$& $=$ &\\
				$\cov(\Mkamuka)$& $=$ &$\cov(\Mmumubd)$& $=$ &$\cov(\Mmumuka)$& $=$ &$\cov(\Mmucsmu)$& $=$ &$\cov(\Mmumumu)$& $=$ &$\kappa^+$
			\end{tabular}
		\end{center}
	\end{crl}

	The remaining two spaces $(\mucs,\kappa)$ and $(\muka,\kappa)$ behave quite differently. For instance, the forcing notions $\Cmucska$ and $\Cmukaka$ are both equivalent to the $\mu$-length ${<}\kappa$-support product of $\kappa$-Cohen forcing $\bb C_\kappa:=\Ckacska$. Therefore, these forcing notions preserve cardinals and thus we cannot apply \cref{lmm:collapse forcing lemma}.
	
	Indeed, it is known that $\cov(\Mmucska)$ is consistently strictly larger than $\kappa^+$.  
	The following theorem is also due to \citet[Theorem 1.7]{Landver92}, with a partial case ($\kappa = \omega$) having been proved earlier by \citet{Miller82}.
	
	\begin{thm}[{\citet{Landver92,Miller82}}]\label{thm:LandverMiller}
		Assume that $\bf V\md\ap{\sf{GCH}}$ and let $G$ be a \mbox{$\Cmucska$-generic} filter over $\bf V$. Then $\bf V[G]\md\ap{\cov(\Mmucska) = \mu < \mu^+ = \cov(\Mkacska)}$.
	\end{thm}
	
	This result invites the following general question, to which we do not know the answer.
	
	\begin{qst}
		For which $\lambda,\nu$ such that $\kappa^+\leq \lambda\leq\nu$ is it consistent that \[\cov(\Mmucska)=\lambda<\nu=\cov(\Mkacska)\,?\] 
	\end{qst}
	
	To conclude this section, we will show that the cardinal characteristics of the spaces $(\mucs,\kappa)$ and $(\muka,\kappa)$ coincide. Thus, it follows that also $\cov(\Mmukaka)=\mu$ is consistent.
	
	\begin{thm}\label{thm:mucs muka same}
		\begin{align*}
			\cov(\Mmucska)&=\cov(\Mmukaka), & \add(\Mmucska)&=\add(\Mmukaka),\\
			\non(\Mmucska)&=\non(\Mmukaka),& 
			\cof(\Mmucska)&=\cof(\Mmukaka).
		\end{align*}
	\end{thm}
	
	To prove the theorem, we describe how to \emph{code} elements of $\muka$ as elements of $\mucs$ so that the coding interacts nicely with the notion of category. We first describe the coding.
	
	\begin{dfn}\leavevmode
		\begin{enumerate}
			\item For $\xi\in\kappa$, let us define $\iota_\xi\colon \xi+1\to2$ by $\iota_\xi(\eta)=0$ if $\eta<\xi$ and $\iota_\xi(\xi)=1$.
			\item If $s\in{}^\gamma\kappa$ for some $\gamma\leq\kappa$, we define $\code(s)\in{}^{\gamma'}2$ as 
			\[\code(s):=\iota_{s(0)}\append \iota_{s(1)}\append\dots\append\iota_{s(\xi)}\append\dots,\]
			which is the concatenation of the sequence $\abset{\iota_{s(\xi)}}{\xi\in\gamma}$, and thus $\gamma'=\sum_{\xi\in\gamma}(s(\xi)+1)$.
			\item For $t\in{}^{\leq\kappa}2$, let us define the function $\decode(t) \in {}^{\leq\kappa}\kappa$ by the $\subseteq$-maximum element $s$ of ${}^{\leq\kappa}\kappa$ such that $\code(s) \subseteq t$. 
			\item For $t\in{}^{\leq\kappa}2$, let us define $\dcset(t)$, which is a set of extensions of $\decode(t)$ associated to $t$. If $\code(\decode(t))=t$, we define $\dcset(t)=\ac{s_t}$, and otherwise we define
			\[\dcset(t)=\set[big]{s_t\cup\ac{(\dom(s_t),\zeta)}}{\ot{\big(\code(\decode(t)) \setminus t\big)} \leq\zeta<\kappa}.\] 
			\item Let a partial function $s\colon \mu\pto\kappa$ (including $s\colon \mu\pto 2$) be called \emph{intact} if for every $\alpha\in\mu$, $\dom(s)\cap [\kappa\alpha,\kappa\alpha+\kappa)$ is an initial segment of the interval $[\kappa\alpha,\kappa\alpha+\kappa)$. 
		\end{enumerate}
	\end{dfn}
	
	The key point is that when $s\colon \mu\pto \kappa$ is intact, we obtain a code $\Code(s)$ by combining codes of $s\restr[\kappa\alpha, \kappa\alpha + \kappa)$ for each $\alpha$. More precisely, let $\pi_\alpha\colon \kappa\to[\kappa\alpha,\kappa\alpha+\kappa)$ be the order-preserving bijection and $s\colon \mu\pto \kappa$ and $t\colon \mu\pto 2$ be intact, then we define $\Code(s)\colon \mu\pto 2$ and $\Decode(s)\colon \mu\pto \kappa$ by 
	\begin{align*}
		\Code(s)=\Cup_{\alpha\in\mu}(\code(s\circ \pi_\alpha)\circ \pi^{-1}_\alpha),&&\Decode(s)=\Cup_{\alpha\in\mu}(\code(t\circ \pi_\alpha)\circ \pi^{-1}_\alpha).
	\end{align*}
	We define the auxiliary function $\Set(t)$ to be the set of partial functions from $\mu$ to $\kappa$ defined by
	\[\Set(t)=\set{\Cup_{\alpha\in\mu}(f(\alpha)\circ \pi^{-1}_\alpha)}{f\in\prod_{\alpha\in\mu}\dcset(t\circ \pi_\alpha)}.\] 
	Thus we can check the following fact, the proof of which we leave to the reader.
	
	\begin{fct}\label{density}\leavevmode
		\begin{enumerate}
			\item\label{density decode} If $\cal D\subset\rm{Fn}_\kappa(\mu,2)$ is such that $D:=\Cup_{t\in \cal D}[t]$ is open dense in $(\mucs,\kappa)$, then
			\[D':=\ts\Cup\set[b]{[s]}{s\in\Set(t)\text{ for some intact }t\in\rm{Fn}_\kappa(\mu,2)\text{ with }[t]\subset D}\]
			is open dense in $(\muka,\kappa)$.
			\item\label{density code} If $\cal D\subset\rm{Fn}_\kappa(\mu,\kappa)$ is such that $D:=\Cup_{s\in \cal D}[s]$ is open dense in $(\muka,\kappa)$, then 
			\[D':=\ts\Cup\set[i]{[\Code(s)]}{s\in\rm{Fn}_\kappa(\mu,\kappa)\text{ is intact and }[s]\subset D}\]
			is open dense in $(\mucs,\kappa)$.\qed
		\end{enumerate}
	\end{fct}
	
	For \cref{thm:mucs muka same}, we will only show $\cov(\Mmukaka) = \cov(\Mmucska)$ and $\cof(\Mmucska) = \cof(\Mmukaka)$. Each proof can be given in the form of a Galois--Tukey connection (cf.\ \cite[\S\,4]{Blass10}), and thus the dual equalities $\non(\Mmukaka)=\non(\Mmucska)$ and $\add(\Mmucska)=\add(\Mmukaka)$ follow as well.
	
	\begin{proof}[Proof of \cref{thm:mucs muka same}]
		We first show that $\cov(\Mmukaka) \leq \cov(\Mmucska)$.
		
		Since $\kappa^+ \leq \cov(\Mmukaka), \cov(\Mmucska)$, it suffices to show that for each $\lambda < \cov(\Mmukaka)$ and family $\mathcal{F}$ of nowhere dense sets in $\mucs$ with $|\mathcal{F}| \leq \lambda$, the family $\mathcal{F}$ does not cover $\mucs$. For each $N \in \mathcal{F}$,
		\[ D_N =\ts\Cup\set{[s]}{s\in\Set(t)\text{ for some intact }t\in\rm{Fn}_\kappa(\mu,2)\text{ with }[t]\cap N = \emptyset},  \]
		is open dense in $\muka$ by \cref{density}\,\eqref{density decode}. Since $|\mathcal{F}| \leq \lambda < \cov(\Mmukaka)$, we see $\{ \muka \setminus  D_N \mid N \in \mathcal{F} \}$ does not cover $\muka$. Take $f \notin \bigcup\{ \muka \setminus  D_N \mid N \in \mathcal{F} \}$ and check that $\Code(f) \notin \bigcup\mathcal{F}$.
		
		The reverse inequality follows in much the same way, swapping the roles of $\Code$ and $\Decode$, and using \cref{density}\,\eqref{density code} instead.
		
		Now let us show that $\cof(\Mmukaka) \leq \cof(\Mmucska)$. It suffices to show that for each $\lambda < \cof(\Mmukaka)$ and family $\mathcal{F}$ of meager sets in $\mucs$ with $|\mathcal{F}| \leq \lambda$, the family $\mathcal{F}$ is not cofinal in $\mucs$. 
		For $X \in \mathcal{F}$, let $\abset{X_\xi}{\xi\in\kappa}$ be a sequence of nowhere dense sets in $\mucs$ such that $X=\Cup_{\xi\in\kappa}X_\xi$. For each $\xi\in\kappa$, we define
		\[D^X_\xi=\ts\Cup\set{[s]}{s\in\Set(t)\text{ for some intact }t\in\rm{Fn}_\kappa(\mu,2)\text{ with }[t]\cap X_\xi = \emptyset}.\]
		Then $X' = \Cup_{\xi\in\kappa} (\muka \setminus D^X_\xi)$ is $\kappa$-meagre in $\muka$ by \cref{density}\,\eqref{density decode}. Since $|\mathcal{F}| \leq \lambda < \cof(\Mmukaka)$, we see that $\{  X' \mid X \in \mathcal{F} \}$ is not cofinal in $\muka$. Thus we can take a $\kappa$-meagre set $Y \subseteq \muka$ such that $Y\nsubseteq X'$ for every $X \in \mathcal{F}$. Let $\abset{Y_\xi}{\xi\in\kappa}$ be a sequence of nowhere dense sets in $\muka$ such that $Y=\Cup_{\xi\in\kappa}Y_\xi$. For each $\xi\in\kappa$, and let us define
		\[D^Y_\xi=\ts\Cup\set{[\Code(s)]}{s\in\rm{Fn}_\kappa(\mu,\kappa)\text{ is intact and }[s]\cap Y_\xi = \emptyset}.\]
		By \cref{density}\,\eqref{density code}, we see that $Y' = \Cup_{\xi\in\kappa} (\muka \setminus D^Y_\xi)$ is $\kappa$-meagre in $\muka$. Finally we note that $Y' \nsubseteq X$ for every $X \in \mathcal{F}$. Therefore, $\mathcal{F}$ is not cofinal in $\mucs$.
		
		We can check the reverse inequality by a similar argument, again swapping the roles of $\Code$ and $\Decode$, and of \cref{density}\,\eqref{density code} and \cref{density}\,\eqref{density decode}.
	\end{proof}

	\section{Uniformity numbers}\label{sec:Uniformity}
	
	In this section we will study the uniformity numbers (cf.\ \cref{dfn:characteristics}) of the $\kappa$-meagre ideals of the spaces from \cref{rmk:spaces}. If $(X,\tau)$ is one of such spaces, then $\MXtau$ is a proper ideal, and thus $X\notin\MXtau$. We therefore find the following rather trivial upper bound.
	
	\begin{thm}\label{thm:uniformity upper bound}
		Let $(X,\tau)$ be a space considered in \cref{rmk:spaces}. Then $\non(\MXtau)\leq \card X$.
	\end{thm}
	
	A lower bound of $\kappa^+$ follows from $\add(\MXtau)\geq\kappa^+$, but this is far from optimal. Indeed many of the spaces have $\mu^+$ or even $\mu^{<\mu}$ as a lower bound. But this does not apply to all of the spaces of \cref{rmk:spaces}. Similar to the situation with the covering numbers, the two spaces $(\mucs,\kappa)$ and $(\muka,\kappa)$ behave quite differently and the uniformity numbers of their $\kappa$-meagre ideals are consistently as small as $\kappa^+$.
	
	We will first provide a topological tool that will help us compare the uniformity numbers of different spaces to each other.
	
	\begin{lmm}\label{lmm:uniformity cont open}
		If $(X,\tau)$ and $(Y,\sigma)$ are spaces, and $\rho\colon X\to Y$ is a continuous open surjection, then $\non(\MXtau)\geq\non(\MYsigma)$.
	\end{lmm}
	\begin{proof}
		Suppose $A\subset X$ and $\card A<\non(\MYsigma)$, then $\rho[A]\in\MYsigma$, and thus there are nowhere dense sets $N_\xi$ for $\xi\in\kappa$ such that $\rho[A]=\Cup_{\xi\in\kappa}N_\xi$. However, $\rho^{-1}[N_\xi]$ is nowhere dense in $(X,\tau)$, since for any nonempty open $U\subset X$, the set $\rho[U]$ is open in $Y$, and thus contains a nonempty open $V\subset\rho[U]$ such that $V\cap N_\xi=\emp$. Then  $\rho^{-1}[V]\cap U$ is a nonempty open that is disjoint from $\rho^{-1}[N_\xi]$. Thus $A\subset\Cup_{\xi\in\kappa}\rho^{-1}[N_\xi]\in\MXtau$.
	\end{proof}

	We are now ready to prove several lower bounds.
	
	\begin{thm}\label{thm:uniformity lower bounds}\leavevmode
		\begin{enumerate}
			\item\label{thm:uniformity lower bounds:1} \begin{enumerate}
				\item $\mu^{<\kappa}\leq \non(\Mkamuka)$,
				\item $2^{<\mu}\leq\non(\Mmucsmu)$,
				\item $\mu^{<\mu}\leq \min\ac{\non(\Mmumubd),\non(\Mmumumu)}$.
			\end{enumerate}
			\item\label{thm:uniformity lower bounds:2} \begin{enumerate}
				\item $\non(\Mkakaka)\leq\non(\Mkamuka)$,
				\item $\non(\Mkamuka)\leq\non(\Mmumubd)$,
				\item $\non(\Mkamuka)\leq\non(\Mmumuka)$, 
				\item $\non(\Mmucsmu)\leq\non(\Mmumumu)$, 
				\item $\non(\Mkakaka)\leq\non(\Mmukaka)$, 
				\item $\non(\Mmukaka)\leq\non(\Mmumuka)$.
			\end{enumerate}
			\item\label{thm:uniformity lower bounds:3} $\cf(\sqb\mu^\kappa)\leq \min\ac{\non(\Mkamuka),\non(\Mmucsmu)}$.
		\end{enumerate}
	\end{thm}
	
	\begin{proof}
		\begin{enumerate}[wide,itemsep=.25\baselineskip]
			\item For $(\kamu,\rm{bd})$, note that every basic open contains a family of $\mu^{<\kappa}$-many mutually disjoint open subsets, hence every $N\subset\kamu$ with $\card N<\mu^{<\kappa}$ is nowhere dense. The same reasoning applies to the other inequalities.
			
			\item We prove each of the relations by providing a continuous open surjection and applying \cref{lmm:uniformity cont open}. Specifically:
			\begin{enumerate}[leftmargin=3\parindent]
				\item Define $\rho_0\colon \kamu\to\kaka$ sending $x\mapsto y:= x\mod \kappa$, or to be precise, for each $\alpha\in\kappa$ there is $\beta\in\mu$ with $x(\alpha)=\kappa\beta+y(\alpha)$. 
				\item Fix a strictly increasing cofinal sequence $\abset{\mu_\xi}{\xi\in\kappa}$ in $\mu$, then we define $\rho_1\colon \mumu\to\kamu$ by sending $x\mapsto y$ where $y(\xi)=x(\mu_\xi)$ for each $\xi\in\kappa$.
				\item Define  $\rho_2\colon \mumu\to\kamu$ sending $x\mapsto x\restr \kappa$.
				\item Define $\rho_3\colon \mumu\to\mucs$ sending $x\mapsto y:=x\mod 2$, or to be precise, for each $\alpha\in\mu$ there is $\beta\in\mu$ with $x(\alpha)=2\beta+y(\alpha)$.
				\item Define  $\rho_4\colon \muka\to\kaka$ sending $x\mapsto x\restr \kappa$.
				\item Define $\rho_5\colon \mumu\to\muka$ sending $x\mapsto x\mod\kappa$.
			\end{enumerate}
			
			\item We first show $\cf(\sqb\mu^\kappa)\leq\non(\Mkamuka)$. Suppose that $X\subseteq \kamu$ with $|X| < \cf ([\mu]^\kappa)$ and define $A = \{\ran (f) \mid f \in X \} \subseteq [\mu]^{\leq \kappa} $. Since $|A| < \cf ([\mu]^\kappa)$, there is a $y \in [\mu]^\kappa$ such that $y \nsubseteq a$ for all $a \in A$, in particular, $y \nsubseteq \ran (f)$ for all $f \in X$. Hence, $X \subseteq \bigcup_{\gamma \in y} \{ f \in \kamu \mid \gamma \notin  \ran (f) \}$. Finally note that $\{ f \in \kamu \mid \gamma \notin  \ran (f) \}$ is nowhere dense in $(\kamu,\rm{bd})$.
			
			Now let us show $\cf(\sqb\mu^\kappa)\leq\non(\Mmucsmu)$.
			Let $X\subset\mucs$ with $\card{X}<\cf(\sqb\mu^\kappa)$. In this proof we will require that $\abset{\mu_\xi}{\xi\in\kappa}$ is cofinal and strictly increasing in $\mu$ and furthermore that $\mu_\xi$ is regular for every $\xi\in\kappa$. We first note that we may assume that each $x\in X$ contains sequences of consecutive $0$'s of length $\mu_\xi$ for any $\xi\in\kappa$. Namely, the set \[C_\xi=\set{x\in\mucs}{\forall \alpha\in\mu\,\exists \gamma\in\iv{\alpha}{\alpha+\mu_\xi}\,(x(\gamma)\neq 0)}\] is nowhere dense for each $\xi<\kappa$, thus we may assume without loss that $X\cap \Cup_{\xi\in\kappa}C_\xi=\emp$.  Consequently, for each $x\in X$ we may define a set $A_x\in\sqb\mu^\kappa$ by $\alpha\in A_x$ if and only if there is $\xi\in\kappa$ such that $\alpha$ is the minimal ordinal for which $x(\gamma)=0$ for all $\gamma\in\iv{\alpha}{\alpha+\mu_\xi}$. We let $\abset[b]{\alpha^x_\xi}{\xi\in\kappa}$ increasingly enumerate $A_x$.
			
			We now gather the order-type of the (non-consecutive) sequence of $1$'s that occur in $x$ in each of the intervals $\iv[b]{\alpha_\xi^x}{\alpha_{\xi+1}^x}$ provided by $A_x$. That is to say, we define \[z_x=\set{\ot\big(\iv[b]{\alpha^x_\xi}{\alpha^x_{\xi+1}}\cap x^{-1}(1)\big)}{\xi\in\kappa}.\] Because we assumed $\card{X}<\cf(\sqb\mu^\kappa)$, there is some $y\in \sqb\mu^\kappa$ with $y\nsubset z_x$ for all $x\in X$. Let us define $X_\beta=\set{x\in\mucs}{\beta\notin z_x}$, then we claim that $X_\beta$ is nowhere dense, and thus that $\Cup_{\beta\in y}X_\beta\in\Mmucsmu$. If $x\in X$, then $y\nsubset z_x$ implies that there is some $\beta\in y\setminus z_x$, and thus $x\in X_\beta$, which shows that $X\subset\Cup_{\beta\in y}X_\beta$.
			
			To see that $X_\beta$ is nowhere dense, let $s\colon \mu\pto 2$ be a partial function with $\card{\dom(s)}<\mu$, then there is some $\xi\in\kappa$ such that $\card{\dom(s)},\beta<\mu_\xi$. Note that $\mu_{\xi+1}$ is regular, so $\eta:=\sup(\dom(s)\cap \mu_{\xi+1})<\mu_{\xi+1}$. Now extend $s$ to $s'$ with $\dom(s')=\dom(s)\cup \mu_{\xi+1}$ by adding a consecutive sequence $\ab 0^{\mu_\xi}\append \ab 1^\beta\append 0^{\mu_{\xi+1}}$ to $s$, for instance by
			\begin{align*}
				s'(\alpha)=\begin{cases}
					s(\alpha)&\text{if }\alpha\in\dom(s),\\
					1&\text{if }\alpha\in\eta\setminus\dom(s)\text{ or }\eta+\mu_\xi\leq \alpha<\eta+\mu_\xi+\beta,\\
					0&\text{for all other }\alpha\in\mu_{\xi+1}.\\
				\end{cases}
			\end{align*} 
			It then follows for any $x\in[s']$ we have $\alpha^x_\xi=\eta$ and $\alpha^x_{\xi+1}=\eta+\mu_\xi+\beta$, and thus $\beta\in z_x$. Therefore $[s']\cap X_\beta=\emp$.
			\qedhere
		\end{enumerate}
	\end{proof}
	
	\begin{thm}\label{crl:mulmu Mmucsmu}
		$\mu^{<\mu}\leq\non(\Mmucsmu)=\non(\Mmumumu)$.
	\end{thm}
	\begin{proof}
		We first show the inequality. Because $\mu^\kappa=2^\kappa\cdot\cf(\sqb\mu^\kappa)$ and both $\cf(\sqb\mu^\kappa)$ and $2^{<\mu}$ are lower bounds of $\non(\Mmucsmu)$, it follows $\mu^\kappa$ is a lower bound. Secondly, since $\mu^{<\mu}=2^{<\mu}\cdot\mu^\kappa$ and also $2^{<\mu}$ is a lower bound, we see that $\mu^{<\mu}\leq\non(\Mmucsmu)$.
		
		Now for the equality, this is clearly the case if  $(\mucs,\mu)$ and $(\mumu,\mu)$ are homeomorphic. On the other hand, if $(\mucs,\mu)$ and $(\mumu,\mu)$ are not homeomorphic, then by \cref{thm:homeo between mumu mu and muka mu,thm:homeo between mucs mu and muka mu} we see that $\mu$ is a strong limit. But, if $\mu$ is a strong limit, then \[2^\mu=(2^{<\mu})^\kappa=\mu^\kappa\leq\non(\Mmucsmu)\leq\non(\Mmumumu)\leq 2^\mu.\qedhere\]
	\end{proof}

	\begin{figure}[t]
		\begin{tikzpicture}[xscale=1.5,yscale=1.2]
			
			\node (k+) at (0,0) {$\kappa^+$};
			\node (cfmuka) at (1,1) {$\cf(\sqb\mu^\kappa)$};
			\node (kamubd) at (2,2) {$\non(\Mkamuka)$};

			\node (kakaka) at (2,0) {$\non(\Mkakaka)$};
			\node (mulka) at (0,2) {$\mu^{<\kappa}$};
			
			\node (mukaka) at (5.5,0) {$\non(\Mmukaka)$};	
			\node (2ka) at (3,1) {$2^\kappa$};

			\node (mucsmu) at (3,5) {\llap{$\non(\Mmucsmu)={}$}$\non(\Mmumumu)$};
			
			\node (muka) at (3,3) {$\mu^{\kappa}$};
			\node (mulmu) at (3,4) {$\mu^{<\mu}$};
			
			\node (mumubd) at (4.5,4) {$\non(\Mmumubd)$};
			
			\node (mumuka) at (5.5,2) {$\non(\Mmumuka)$};
			\node (2mu) at (5.5,5) {$2^{\mu}$};

			\draw (k+) edge[->] (kakaka);
			\draw (k+) edge[->] (cfmuka);
			\draw (cfmuka) edge[->] (kamubd);
			\draw (mulmu) edge[->] (mucsmu);
			\draw (k+) edge[->] (mulka);
			\draw (mulka) edge[->] (kamubd);
			\draw (mulmu) edge[->] (mumubd);
			\draw (kakaka) edge[->] (2ka);
			\draw (2ka) edge[->] (muka);
			\draw (kakaka) edge[->] (kamubd);
			\draw (kamubd) edge[->] (mumuka);
			\draw (kakaka) edge[->] (mukaka);
			\draw (kamubd) edge[->] (muka);
			\draw (muka) edge[->] (mulmu);
			\draw (mumuka) edge[->] (2mu);
			\draw (mumubd) edge[->] (2mu);
			\draw (mucsmu) edge[->] (2mu);
			\draw (mukaka) edge[->] (mumuka);
			
		\end{tikzpicture}
		
		\caption{Summary of \cref{thm:uniformity upper bound,thm:uniformity lower bounds,crl:mulmu Mmucsmu}.}
		\label{fig:cardinals diagram}
	\end{figure}
	
	We saw in \cref{thm:mucs muka same} that $\non(\Mmucska)=\non(\Mmukaka)$. Contrary to the other spaces, the uniformity numbers of the $\kappa$-meagre ideals of $(\mucs,\kappa)$ and $(\muka,\kappa)$ can be consistently smaller than $\mu$. For this result, we require the combinatorial concept of \emph{independent families}.
	
	\begin{dfn}
		Let $\lambda$ be a regular infinite cardinal. A family $\cal X\subset[\lambda]^\lambda$ is \emph{${<}\nu$-independent}, if for every $A,B\in[\cal X]^{<\nu}$ with $A\cap B=\emp$ we have $\card{\Cap A\setminus \Cup B}=\lambda$.
	\end{dfn}
	
	\begin{thm}\label{mu < non Mmucska}\leavevmode
		\begin{enumerate}
			\item If $\lambda<2^\lambda<\mu$, then $\lambda<\non(\Mmucska)$. Hence, if $\mu$ is strong limit, then $\mu^+\leq\non(\Mmucska)$.
			\item If $\kappa<\lambda<\mu\leq 2^\lambda$ and there exists a ${<}\kappa^+$-independent family $\cal X\subset[\lambda]^\lambda$ of cardinality $\mu$, then $\non(\Mmucska)\leq\lambda$.
		\end{enumerate}
	\end{thm}
	\begin{proof}
		\begin{enumerate}[wide,label={\textit{(\arabic*)}}]\item
			Assume that $\kappa<\lambda<2^\lambda<\mu$. Let $\abset{f_\eta}{ \eta\in\lambda}$ be a sequence in $\mucs$ and define $g_\xi\in\lacs$ by $g_\xi\colon \eta\mapsto f_\eta(\xi)$ for each $\eta\in\lambda$ and $\xi\in\mu$. Suppose $s\colon \mu\pto2$ and $\card{\dom(s)}<\kappa$. Since $2^\lambda<\mu$, there are two distinct $\xi,\xi'\in\mu\setminus\dom(s)$ such that $g_\xi=g_{\xi'}$. Let $s'\supset s$ be an extension such that $s'(\xi)=0$ and $s'(\xi')=1$, then $f_\eta\notin [s']$ for all $\eta\in\lambda$, because $f_\eta(\xi)=g_\xi(\eta)=g_{\xi'}(\eta)=f_\eta(\xi')$ whereas $s'(\xi)\neq s'(\xi')$. Therefore $\set{f_\eta}{ \eta\in\lambda}$ is nowhere dense, hence every $X\subset\mucs$ with $|X|\leq\lambda$ is nowhere dense.
			
			If $\mu$ is strong limit and $X\subset\mucs$ with $\card X=\mu$, there is a sequence $\abset{\mu_\xi}{\xi\in\kappa}$ cofinal in $\mu$ and  $X_\xi$ such that $X=\Cup_{\xi<\kappa}X_\xi$ with $\card{X_\xi}=\mu_\xi<2^{\mu_\xi}<\mu$, and thus by the above $X\notin \Mmucska$.
			
			\item 
			Let $\kappa<\lambda<\mu\leq 2^\lambda$ and let $\cal X=\set{X_\xi}{ \xi\in\mu}\subset[\lambda]^\lambda$ be a ${<}\kappa^+$-independent family of cardinality $\mu$ and let $g_\xi$ be the characteristic functions of $X_\xi$ for each $\xi\in\mu$. Define $f_\eta\in\mucs$ by $f_\eta(\xi)=g_\xi(\eta)$ for each $\eta\in\lambda$ and $\xi\in\mu$. We claim that $\cal F=\set{f_\eta}{ \eta\in\lambda}\notin\Mmucska$. 
			
			Suppose that $\abset{\cal D_\alpha}{ \alpha\in\kappa}$ is a sequence of open dense subsets of $\mucs$. Recursively define a sequence $\abset{s_\alpha}{ \alpha\in\kappa}$ such that $s_\alpha\colon \mu\pto 2$ with $\card{\dom(s_\alpha)}<\kappa$, $[s_\alpha]\subset D_\alpha$ and $s_\alpha\subset s_{\alpha'}$ for all $\alpha<\alpha'$. Let $s=\Cup_{\alpha\in\kappa}s_\alpha$, then $s\colon \mu\pto2$ with $\dom(s)\leq\kappa$. Since $\cal X$ is a ${<}\kappa^+$-independent family, there is some $\alpha\in \Cap_{\xi\in s^{-1}(1)}X_\xi\setminus \Cup_{\xi\in s^{-1}(0)}X_\xi$. It then follows that $s(\xi)=g_\xi(\alpha)=f_\alpha(\xi)$ for each $\xi\in\dom(s)$, thus $f_\alpha\in D_\alpha$ for each $\alpha\in\kappa$. In other words, $\cal F$ intersects every intersection of $\kappa$-many open dense sets, and thus is not $\kappa$-meagre.\qedhere
		\end{enumerate}
	\end{proof}
	
	Let us now move to the topic of consistent values for the uniformity numbers. We first use the combinatorics of the regular $\kappa$-meagre ideal to show that any of the uniformity numbers can have any regular value above $\mu$. Simultaneously this shows that each of the three lower bounds $\non(\Mkakaka)$,  $\cf(\sqb\mu^\kappa)$ and $\mu^{<\kappa}$ of $\non(\Mkamuka)$ may be strict.
	
	\begin{thm}\label{thm:regular values for non kakaka}
		Assume $\sf{GCH}$ and let $\lambda>\kappa$ be regular, then there is a cardinal-preserving forcing extension in which  $\mu^{<\kappa}<\cf([\mu]^\kappa)=\mu^+$ and $\non(\Mkakaka)=2^\kappa=\lambda$ and $2^\mu=2^\kappa\cdot\mu^+$.
	\end{thm} 
	\begin{proof}
		\citet{CummingsShelah95} showed that $\fr b_\kappa=2^\kappa=\lambda$ can be achieved through a $\lambda$-length iteration of $\kappa$-Hechler forcing with ${<}\kappa$-support. Since $\fr b_\kappa\leq\non(\Mkakaka)$, this iteration is the required forcing notion. Note that this iteration is both ${<}\kappa^+$-c.c.\ and ${<}\kappa$-closed and hence preserves cardinals and cofinalities.\footnote{See also \cite[\S\,4.2]{BrendleBrookeTaylorFriedmanMontoya18}.} The chain condition also implies that every $A\in\sqb\mu^\kappa$ in the forcing extension is contained in some $B\in\sqb\mu^\kappa$ from the ground model, and thus $\cf(\sqb\mu^\kappa)=\mu^+$ is witnessed by the ground model set $(\sqb\mu^\kappa)^\bf V$. On the other hand, ${<}\kappa$-closure implies that no sets of ordinals of size ${<}\kappa$ are added, thus $\mu^{<\kappa}=\mu$ holds in the extension.
	\end{proof}

	\begin{thm}
		Assume $\sf{GCH}$ and let $\lambda,\nu$ be regular cardinals with $\kappa<\lambda\leq\nu<\mu$, then there is a cardinal-preserving forcing extension in which $\non(\Mkakaka)=\lambda$ and $\non(\Mmucska)=\nu$.
	\end{thm}
	\begin{proof}
		We first use \cref{thm:regular values for non kakaka} to get a model where $\non(\Mkakaka)=2^\kappa=\lambda$. Take the resulting model as our ground model $\bf V$. Let $\bb P$ be the $\mu^+$-length ${<}\nu$-support product of $\nu$-Cohen forcing $\bb C_\nu:=\Cnucsnu$ (in the notation of \cref{lmm:collapse forcing lemma}). If $G$ is $\bb P$-generic over $\bf V$, then in $\bf V[G]$ it follows that $\nu$ is minimal with the property that $2^\nu\geq\mu$. Since $\bb P$ is ${<}\nu$-closed, it is certainly ${<}\lambda$-closed as well. Therefore we see that $([\kaka]^{<\lambda})^{\bf V[G]}=([\kaka]^{<\lambda})^\bf V$, and consequently every subset of $\kaka$ of size ${<}\lambda$ is $\kappa$-meagre in $\bf V[G]$, that is, $\bf V[G]\md\ap{\non(\Mkakaka)=\lambda}$.
		
		Similar to classical Cohen forcing, $\bb P$ is forcing equivalent to the $2$-step iteration $\bb P\ast \dot{\bb C}_\nu$. Thus, if $G$ is $\bb P$-generic over $\bf V$, then there exist $G_1,G_2$ such that $G_1$ is $\bb P$-generic over $\bf V$ and $G_2$ is $\bb C_\nu$-generic over $\bf V[G_1]$. We work in $\bf V[G_1]$ and describe a forcing notion $\bb Q$ that is equivalent to $\bb C_\nu$.\footnote{To see this, it suffices to observe that $\bb Q$ is ${<}\nu$-closed, separative, and has cardinality $\nu$. Any such forcing notion is equivalent to $\nu$-Cohen forcing, similar to how any nontrivial countable forcing notion is equivalent to $\omega$-Cohen forcing.}
		The conditions of $\bb Q$ are pairs $(C,\gamma)$ where $\gamma\in\nu$ and $C\subset \cal P(\gamma)$, ordered by $(D,\delta)\leq(C,\gamma)$ if and only if
		\begin{enumerate}
			\item $\gamma\leq\delta$,
			\item for each $c\in C$ there is $d\in D$ such that $c\subset d$,
			\item $d\cap \gamma\in C$ for each $d\in D$.
		\end{enumerate}
		If $H$ is a $\bb Q$-generic filter, then for each $\gamma\in\nu$ there is a unique $C\subset\cal P(\gamma)$ such that $(C,\gamma)\in H$, thus let us write $C_\gamma$ for such $C$. Let $\cal X$ be the family of subsets $x\subset\nu$ such that for every $\gamma\in\nu$ we have $x\cap \gamma\in C_\gamma$. It is easy to see that $\card{\cal X}=2^{\nu}\geq \mu$. We claim that $\cal X$ is ${<}\kappa^+$-independent, which implies that $\bf V[G]=\bf V[G_1][G_2]\md\ap{\non(\Mmucska)=\nu}$ by \cref{mu < non Mmucska}.
		
		Namely, if $A,B\in\sqb{\cal X}^\kappa$ are disjoint, then by regularity of $\nu$ there is some $\gamma\in\nu$ such that $x\cap \gamma\neq x'\cap \gamma$ for all distinct $x,x'\in A\cup B$. For each $\alpha_0\in\nu$, the set of conditions $(D,\delta)$ that force that there exists an $\alpha\geq\alpha_0$ for which $\alpha\in\Cap A\setminus\Cup B$ is dense below $(C_\gamma,\gamma)$.
	\end{proof}
	
	For the space $(\kamu,\rm{bd})=(\kamu,\kappa)$ we have the upper bound $\non(\Mkamuka)\leq \mu^\kappa$. We will show that this upper bound is consistently strict when we assume that $\kappa$ is uncountable.
	
	\begin{thm}
		Assume $\sf{GCH}$ and that $\kappa$ is uncountable, and let $\lambda>\mu$ be a cardinal with $\cf(\lambda)>\kappa$. Then there is a cardinal-preserving forcing extension in which $\non(\Mmumuka)\leq\mu^+\leq\mu^\kappa=\lambda$.
	\end{thm}
	
	\begin{proof}
		First we force with (classical) Cohen forcing to obtain a model of $2^{\aleph_0}=\mu$. We will take this model as our ground model $\bf V$ and fix a bijection $b\colon\omcs\to\mu$. Now let $\bb P$ be the $\lambda$-length ${<}\kappa$-support product of $\kappa$-Cohen forcing $\bb C_\kappa$. Since $\bb P$ is ${<}\kappa$-closed (and thus $\sigma$-closed) and does not add new ordinals, it follows that $b$ is still a bijection in the extension by $\bb P$. This bijection lifts to a bijection $b^*\colon{}^\kappa(\omcs)\to\kamu$. This allows us to canonically reinterpret the first $\mu^+$-many $\kappa$-Cohen generics $\abset{x_\xi}{\xi\in\mu^+}$: We have $x_\xi\colon \kappa\to 2$, and thus by splitting $x_\xi$ into blocks of length~$\omega$, we may in fact view $x_\xi$ as a function $\kappa\to\omcs$, and thus $x_\xi^*:=b^*(x_\xi)$ is a function from $\kappa$ to $\mu$. Now, for each $\alpha\in\mu^+$ we define $y_\alpha\colon\mu\to\mu$ by concatenating blocks of $\mu$-many $x_\xi^*$'s. To be precise, for each $\beta\in\mu$,  $y_\alpha\restr [\kappa\beta,\kappa\beta+\kappa)$ coincides with $x^*_{\mu\alpha+\beta}$, or to be even more precise, if $\gamma\in\kappa$, then $y_\alpha(\kappa\beta+\gamma)=x_{\mu\alpha+\beta}^*(\gamma)$. Let $Y=\set{y_\alpha}{\alpha\in\mu^+}$, then we claim that $Y\notin\Mmumuka$.
		
		Working in the extension of $\bb P$, let $\abset{D_\xi}{\xi\in\kappa}$ be a sequence of open dense subsets of $(\mumu,\kappa)$. Each $D_\xi$ has some $S_\xi\subset\rm{Fn}_\kappa(\mu,\mu)$ such that $D_\xi=\Cup_{s\in S_\xi}[s]$, and since $\rm{Fn}_\kappa(\mu,\mu)$ is the same set as in the ground model (by ${<}\kappa$-closure) and has cardinality $\mu$, each $D_\xi$ is coded by some $f_\xi\in\mucs$. If $A_{\xi,\eta}\subset\bb P$ is a maximal antichain such that each $p\in A_{\xi,\eta}$ decides the value of $f_{\xi}(\eta)$, then $\card{A_{\xi,\eta}}\leq\kappa$ by the ${<}\kappa^+$-chain condition of $\bb P$, and thus $\card{\Cup_{\xi\in\kappa,\eta\in\mu}A_{\xi,\eta}}\leq\mu$. This implies that there is some $\alpha\in\mu^+$ such that no $p\in\Cup_{\xi\in\kappa,\eta\in\mu}A_{\xi,\eta}$ has a support intersecting with $[\mu\alpha,\mu\alpha+\mu)$. This makes all of the $x_{\mu\alpha+\beta}$ with $\beta\in\mu$ mutually $\kappa$-Cohen generic over $\bf V[\abset{D_\xi}{\xi\in\kappa}]$. 
		It is now an easy argument by $\kappa$-Cohen genericity to show that $y_\alpha\in \Cap_{\xi\in\kappa}D_\xi$, and thus that $Y$ intersects every $\kappa$-comeagre set of $(\mumu,\kappa)$. 
	\end{proof}
	
	We do not know if a similar theorem is provable for $\mu$ a  singular cardinal with countable cofinality.
	
	\begin{qst}
		If $\kappa=\aleph_0$, is $\non(\Mmumuka)<\mu^\kappa$ consistent? Or $\non(\Mkamuka)<\mu^\kappa$?
	\end{qst}

	\section{Dominating and unbounding numbers}\label{sec:Dominating}
	
	Classically, no treatment of the cardinal characteristics of the meagre ideal would be complete without mentioning the bounding and dominating numbers. This section is added mostly for the sake of completeness, and because we will need dominating numbers in \cref{sec:Cofinality}.

	We are primarily interested in domination on the sets of functions $\kaka$, $\kamu$, $\muka$ and $\mumu$. Let us assume $X$ is one of these four sets of functions and that $\delta$ is the domain of the functions in $X$, and let $f,g\in X$, then we define three orderings:
	\begin{enumerate}
		\item $f\leq_\rm{bd}g$ if the set $\set{\alpha\in\delta}{ g(\alpha)<f(\alpha)}$ is bounded in $\delta$. 
		\item$f\leq_\nu g$ if $|\set{\alpha}{ g(\alpha)<f(\alpha)}|<\nu$, where $\nu$ is a cardinal.
		\item$f\leq_\rm{all} g$ if $f(\alpha)\leq g(\alpha)$ for all $\alpha\in\delta$.
	\end{enumerate} 
	
	If $\leq_{\relplaceholder}$ is one of the orderings defined above, then we define the dominating number $\dXph$ as the least size of a set $D\subset X$ such that for each $f\in X$ there exists $g\in D$ with $f\leq_{\relplaceholder} g$. Dually we define the bounding number $\bXph$ as the least size of a set $B\subset X$ such that no $g\in X$ exists for which all $g\in B$ satisfy $f\leq_{\relplaceholder} g$.\footnote{Of course we implicitly assume this definition makes sense in the first place and define $\bXph=\Ord$ if the definition does not make sense, such as for $\bXnu$ where the domain $\delta$ of the elements of $X$ has $\delta<\nu$.} We note that clearly the following implications hold for any $f,g\in X$: 
	\begin{align*}
		f\leq_\rm{all} g\text{ implies }f\leq_\kappa g,\qquad
		f\leq_\kappa g\text{ implies }f\leq_\rm{bd}g,\qquad
		f\leq_\rm{bd} g\text{ implies }f\leq_\mu g.
	\end{align*}
	Therefore, we also see that $\bXal\leq\bXka\leq\bXbd\leq\bXmu$ and  $\dXmu\leq\dXbd\leq\dXka\leq\dXal$. In case the domain $\delta$ of the elements of $X$ is $\kappa$, the relation $\leq_\mu$ is trivial (i.e., equal to $X\times X$) and since $\kappa$ is regular, $\leq_\kappa$ and $\leq_\rm{bd}$ coincide. We also note that we may replace the range of the functions in $X$ with any set of the same cofinality and obtain the same cardinal characteristics, e.g., these domination relations give the same cardinal characteristics on the spaces $\kamu$ and $\kaka$; and similarly between $\mumu$ and~$\muka$. We will therefore purely consider the spaces $\kaka$ and $\muka$ from now on. 
	
	The following lemma shows that the choice of relation does not affect the dominating numbers.
	
	\begin{lmm}[{\cite[Theorem 3.13]{CardonaMejia25}, \cite[Proposition 12]{Brendle22}}]\label{dom all lambda}
		$\lambda^+\leq\dlakala=\dlakaal$ for $\lambda\geq\kappa$.\qed
	\end{lmm}
	
	On the other hand, the bounding numbers on the space $\muka$ happen to all be absolute (and either equal to $\kappa$ or $\kappa^+$. This was proved in the Master's thesis of the first author, but we will repeat the argument here, as it seems to not exist in the literature.
	
	\begin{thm}[\cite{Hayashi23}]
		$\bmukaka=\bmukaal=\kappa$ and $\bmukamu=\bmukabd=\kappa^+$.
	\end{thm}
	\begin{proof}
		For the first pair of equalities, note that the family of constant functions has size $\kappa$ and is unbounded for both the $\leq_\rm{all}$ and the  $\leq_\kappa$  relations.
		
		For the second pair of equalities, let $\abset{\mu_\xi}{\xi\in\kappa}$ be a continuous increasing sequence of cardinals that is cofinal in $\mu$, with $\mu_0=0$. It suffices to prove $\kappa^+\leq\bmukabd$ and $\bmukamu\leq\kappa^+$. For the first inequality, suppose that $\abset{f_\zeta}{\zeta\in\kappa}$ is a sequence in $\muka$. Then we define $g\in\muka$ by $g(\alpha)=\sup_{\zeta\leq\xi}f_\zeta(\alpha)+1$ for each $\alpha\in[\mu_\xi,\mu_{\xi+1})$, and note that $f_\zeta\leq_\rm{bd}g$ for all $\zeta\in\kappa$.
		
		For the second inequality, we fix an injection $\sigma_\eta\colon \eta\to\kappa$ for each $\eta\in\kappa^+$ and define $f_\zeta\in\muka$ for each $\zeta\in\kappa^+$ as follows:
		\[
		\text{For each }\alpha\in\mu\text{ and }\eta\in\kappa^+, \quad f_\zeta(\kappa^+\alpha+\eta):=\begin{cases}
			\sigma_\eta(\zeta) & \text{ if }\zeta<\eta,\\
			0 & \text{ otherwise.}
		\end{cases}
		\]
		Now $\cal F=\set{f_\zeta}{\zeta\in\kappa^+}$ is unbounded for the $\leq_\mu$ relation. Namely, if $g\in\muka$ is a function such that $f_\zeta\leq_\mu g$ for all $\zeta\in\kappa^+$, then by the pigeonhole principle there is some $\xi\in\kappa$ and $Z\subset\kappa^+$ with $\card Z=\kappa^+$ such that $\card[b]{\set{\alpha\in\mu}{g(\alpha)<f_\zeta(\alpha)}}\leq\mu_\xi$ for all $\zeta\in Z$. Since $\kappa^+\cdot \mu_\xi<\mu$, we may fix some $\alpha\in\mu$ and $\eta\in\kappa^+$ with $\card{Z\cap\eta}= \kappa$ such that $f_\zeta(\kappa^+\alpha+\eta)\leq g(\kappa^+\alpha+\eta)$ for all $\zeta\in Z$. Now note that $f_\zeta(\kappa^+\alpha+\eta)=\sigma_\eta(\zeta)$ for all $\zeta\in Z\cap \eta$, and since $\sigma_\eta$ is injective, this implies that $\sup_{\zeta\in Z}f_\zeta(\kappa^+\alpha+\eta)=\kappa$. But this is a contradiction, since $g(\kappa^+\alpha+\eta)<\kappa$.
	\end{proof}
	
	We may summarise the relations between the dominating and bounding numbers on $\kaka$ and $\kamu$ as in \cref{fig:dominating numbers}, displaying the bounding and dominating numbers on $\kaka$ in the bottom row, and the bounding and dominating numbers on $\muka$ in the top row.

	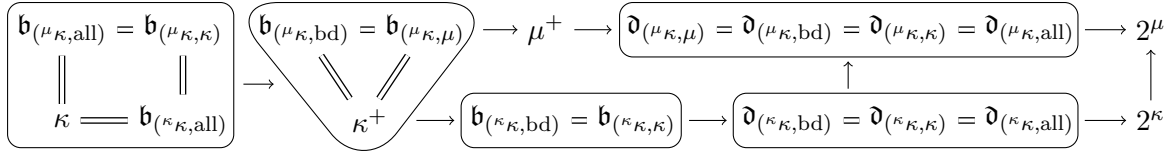
\begin{figure}[t]
		
		\begin{tikzpicture}[xscale=1.6,yscale=1.2]
			\node (k) at (0,0) {$\kappa$};
			\node (k+) at (2.5,0) {\hspace*{4pt}$\kappa^+$};
			\node (bkakaal) at (1,0) {$\bkakaal$};
			\node (bkakabd) at (3.75,0) {$\bkakabd$};
			\node (bkakaka) at (4.75,0) {$\bkakaka$};
			\node (dkakabd) at (6,0) {$\dkakabd$};
			\node (dkakaka) at (7,0) {$\dkakaka$};
			\node (dkakaal) at (8,0) {$\dkakaal$};
			\node (bmukaal) at (0,1) {$\bmukaal$};
			\node (bmukaka) at (1,1) {$\bmukaka$};
			\node (bmukabd) at (2,1) {$\bmukabd$};
			\node (bmukamu) at (3,1) {$\bmukamu$};
			\node (dmukamu) at (5,1) {$\dmukamu$};
			\node (dmukabd) at (6,1) {$\dmukabd$};
			\node (dmukaka) at (7,1) {$\dmukaka$};
			\node (dmukaal) at (8,1) {$\dmukaal$};
			\node (2k) at (9,0) {$2^\kappa$};
			\node (2m) at (9,1) {$2^\mu$};
			\node (m+) at (4,1) {$\mu^+$};
			
			\draw (k) edge[line width=2.5pt] (bkakaal);			
			\draw (k) edge[line width=2.5pt] (bmukaal);
			\draw (bmukaal) edge[line width=2.5pt] (bmukaka);
			\draw (bkakaal) edge[line width=2.5pt] (bmukaka);
			\draw (k+) edge[line width=2.5pt] (bmukabd);
			\draw (k+) edge[line width=2.5pt] (bmukamu);
			\draw (bmukabd) edge[line width=2.5pt] (bmukamu);
			
			\draw (1.5,.4) edge[->] (1.75,.4);
			\draw (k+) edge[->, shorten >=2pt, shorten <=7pt] (bkakabd);
			\draw (bmukamu) edge[->, shorten <=3pt] (m+);
			\draw (m+) edge[->, shorten >=2pt] (dmukamu);
			\draw (bkakabd) edge[line width=2.5pt] (bkakaka);
			
			\draw (bkakaka) edge[->, shorten >=2pt, shorten <=1pt] (dkakabd);
			\draw (dkakabd) edge[line width=2.5pt] (dkakaka);
			\draw (dkakaka) edge[line width=2.5pt] (dkakaal);
			\draw (dmukamu) edge[line width=2.5pt] (dmukabd);
			\draw (dmukabd) edge[line width=2.5pt] (dmukaka);
			\draw (dmukaka) edge[line width=2.5pt] (dmukaal);
			
			\draw (2k) edge[->] (2m);
			\draw (dkakaal) edge[->] (2k);
			\draw (dmukaal) edge[->] (2m);
			\draw (6.5,.35) edge[->] (6.5,.65);
			
			\draw (k) edge[line width=1.5pt,white] (bkakaal);			
			\draw (k) edge[line width=1.5pt,white] (bmukaal);
			\draw (bmukaal) edge[line width=1.5pt,white] (bmukaka);
			\draw (bkakaal) edge[line width=1.5pt,white] (bmukaka);			
			\draw (k+) edge[line width=1.5pt,white] (bmukabd);
			\draw (k+) edge[line width=1.5pt,white] (bmukamu);
			\draw (bmukabd) edge[line width=1.5pt,white] (bmukamu);			
			\draw (bkakabd) edge[line width=1.5pt,white] (bkakaka);
			\draw (dkakabd) edge[line width=1.5pt,white] (dkakaka);
			\draw (dkakaka) edge[line width=1.5pt,white] (dkakaal);
			\draw (dmukamu) edge[line width=1.5pt,white] (dmukabd);
			\draw (dmukabd) edge[line width=1.5pt,white] (dmukaka);
			\draw (dmukaka) edge[line width=1.5pt,white] (dmukaal);

			\draw[rounded corners=5pt] (-.45,-.3) rectangle (1.42,1.3);
			\draw[rounded corners=20pt] (2.5,-.6) -- (1.35,1.3) -- (3.6,1.3) -- cycle;
			\draw[rounded corners=5pt] (3.3,-.3) rectangle (5.12,0.3);
			\draw[rounded corners=5pt] (5.55,-.3) rectangle (8.4,0.3);
			\draw[rounded corners=5pt] (4.61,0.7) rectangle (8.4,1.3);
			
		\end{tikzpicture}
		\caption{Relations between dominating and bounding numbers on $\kaka$ and $\muka$.}\label{fig:dominating numbers}

	\end{figure}
	
	\begin{rmk}
		As shown in \cref{fig:dominating numbers}, of the cardinal characteristic we are interested in, only three remain: $\bkakabd$, $\dkakabd$ and $\dmukabd$. We will abbreviate these as $\fr b_\kappa$, $\fr d_\kappa$ and $\fr d_\mu$, conform with the standard notation for these cardinal characteristics.
	\end{rmk}
	
	Let us discuss the consistency of strict relations allowable in \cref{fig:dominating numbers}. Clearly $\kappa<\kappa^+$ holds in every model, and since $\mu^+\leq\fr d_\mu$ (cf.\ \cref{dmuka cofinality}) we also see that $\kappa^+<\fr d_\mu$ holds in every model and that $2^\kappa<\fr d_\mu$ is consistent. For the bottom row of cardinal characteristics, applying Hechler forcing to the generalised context was done by \citet{CummingsShelah95}:
	
	\begin{thm}[{\cite{CummingsShelah95}}]
		Assume $\sf{GCH}$. If $\beta,\delta,\gamma$ are cardinals such that $\kappa^+\leq\beta\leq\delta\leq\gamma$ and $\beta$ is regular and $\kappa^+\leq\cf(\gamma)$, then there exists a cardinal preserving forcing notion that forces $\fr b_\kappa=\beta$, $\fr d_\kappa=\delta$ and $2^\kappa=\gamma$.
	\end{thm}
	
	Finally, the consistency of $\fr d_\mu<2^\mu$ is a partial open problem and has been solved under certain assumptions.
	
	\begin{thm}[{\cite[Claim 1.5]{Shelah19}}]
		If $\nu^\kappa<\mu$ for all $\nu<\mu$, then $\fr d_\mu=2^\mu$.
	\end{thm}
	
	\begin{thm}[cf. {\cite[Theorem~5.2]{Hayashi26}}]
		If $\kappa$ is uncountable, consistently $\fr d_\mu<2^\mu$.
	\end{thm}
	
	We conclude with another result of the first author, relating $\fr d_\mu$ to the cardinal characteristics from \cref{dfn:cf covering numbers}.
	
	\begin{thm}[{\cite[Theorem~4.1]{Hayashi26}}]\label{dmuka cofinality}
		$\cf([\mu]^\kappa)\leq\fr d_\mu$.
	\end{thm}

	\section{Cofinality numbers}\label{sec:Cofinality}
	
	In the context of higher Baire spaces for regular cardinals (i.e., $\kaka$) it was shown in \cite{Brendle22} that $2^\kappa<\cof(\Mkakaka)$ is consistent, and in fact that $2^{<\kappa}<\cof(\Mkakaka)$ holds for every regular $\kappa$. The proof relies on a lower bound derived from a specific dominating number, and on closer examination it becomes clear that antichains of size $2^{<\kappa}$ in the topology on $\kaka$ form the key to this lower bound. The method of proof can be generalised to give lower bounds for $\cof(\MXtau)$ for our spaces $(X,\tau)$ with $\tau$ either the ${<}\mu$-box topology or the bounded topology. We will use the notation from \cref{sec:Dominating} to describe the relevant dominating numbers.

	\begin{thm}\label{dom < cof bd top}
		\leavevmode
		\begin{enumerate}
			\item
			$\dnukaal\leq \cof(\Mkamuka)$ for $\nu=\mu^{<\kappa}$.
			\item
			$\dnukaal\leq \cof(\Mmumubd)$ for $\nu=\mu^{<\mu}$.
		\end{enumerate}
	\end{thm}
	
	\begin{proof}
		
		The proof structure is the same in both case, so we will discuss the general structure first. For both $(\kamu,\rm{bd})$ and $(\mumu,\rm{bd})$, note that $\nu$ is the weight of the space. As we saw in \cref{lem:weight sized open partition}, these spaces have an open partition $\cal S$ of size $\nu$. We will fix an enumeration $\abset{\sigma_\gamma}{\gamma\in\nu}$ of $\cal S$.
		
		Let $X\in\ac{\kamu,\mumu}$ be the relevant space and suppose $\cal F\subset\MXbd$ has cardinality $\card{\cal F}<\dnukaal$, then we we will define for each $A\in\cal F$ a set of functions $D_A\subset\nuka$ with $\card{D_A}\leq\nu$. Since $\nu<\dnukaal$, and therefore since $\card{\Cup_{A\in\cal F}D_A}\leq\card{\cal F}\cdot \nu<\dnukaal$, we may find some $g\in\nuka$ such that $g\nleq_\rm{all} f$ for any $f\in\Cup_{A\in\cal F}D_A$. We then use $g$ to define a set $C\subset X$ that is nowhere dense in $(X,\rm{bd})$ such that $C\nsubset A$ for all $A\in\cal F$, thereby proving that $\cal F$ does not witness $\cof(\MXbd)$. The proof thus consists of three parts: 
		\begin{enumerate}[label=(\Roman*)]
			\item Construct $D_A$ from $\cal F$,
			\item Construct $C$ from $g$,
			\item Show that $C\nsubset A$ for all $A\in\cal F$.
		\end{enumerate}
		
		Let us now look at each of the spaces in detail.
		
		\begin{enumerate}[label={\textit{(\arabic*)}},wide]
			\item We work in the space $(\kamu,\rm{bd})$, where $\fkamu$ generates the base and $\nu=\mu^{<\kappa}$. We will go through the parts (I) to (III).
			
			(I)\quad
			Let $A\in\cal F$, then we can find a $\subset$-increasing sequence $\abset{A_\alpha}{\alpha\in\kappa}$ of nowhere dense sets such that $A=\Cup_{\alpha\in\kappa}A_\alpha$. We define a function $h_A\colon\fkamu\to\fkamu$ such that if $s\in{}^\alpha\mu$ for $\alpha<\kappa$, then $A_\alpha\cap [s\append h_A(s)]=\emp$. Given $s\in\fkamu$, define $f^A_s\in\nuka$ to map $\gamma\mapsto\dom(h_A(s\append\sigma_\gamma))$ for each $\gamma\in\nu$.
			We let $D_A=\set{f^A_s}{s\in\fkamu}$, then indeed $\card{D_A}\leq\nu$. See also \cref{figure}.

			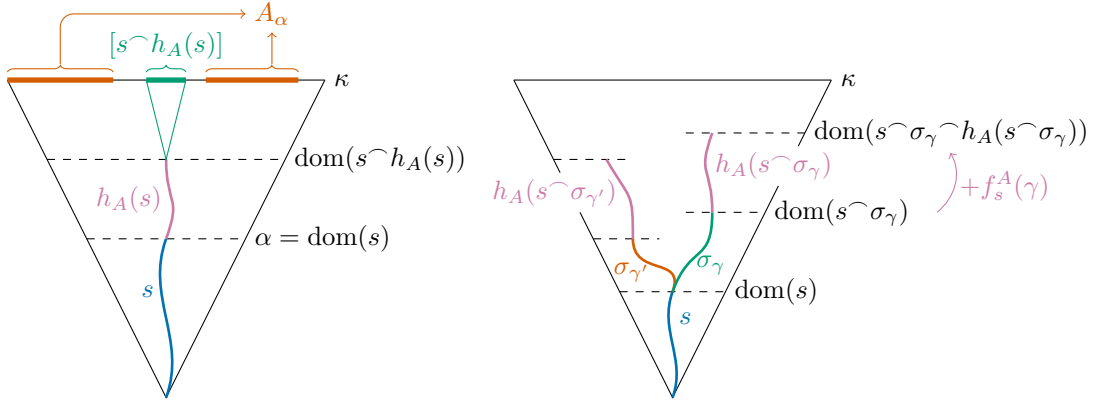
\begin{figure}[h]\small
				\begin{tikzpicture}[scale=.35]
					
					\draw (-6,12) -- (0,0) -- (6,12);
					\draw (-3,6) edge[dashed] (3,6);
					\draw (-4.5,9) edge[dashed] (4.5,9);
					\draw (-6,12) edge (6,12);
					
					\draw (0,0) edge[line width=1pt, in=-110, out=70, niceblue] (0,6);
					\draw (0,6) edge[line width=1pt, in=-90, out=70, looseness=1.5, nicepink] (0,9);
					
					\draw (-6,12) edge[line width=2pt,nicered] (-2,12);
					\draw (1.5,12) edge[line width=2pt,nicered] (5,12);
					\draw[nicered,decorate, decoration = {brace,raise=3pt}] (-6,12) -- (-2,12);
					\draw[nicered,decorate, decoration = {brace,raise=3pt,aspect=.713}] (1.5,12) -- (5,12);
					\draw (4,12.6) edge[nicered,->] (4,13.8);
					\path[draw,nicered,->,rounded corners=5pt] (-4,12.6) -- (-4,14.5) -- (3,14.5);
					
					\draw (-.75,12) edge[line width=2pt,nicegreen] (.75,12);
					\draw[nicegreen,decorate, decoration = {brace,raise=3pt}] (-.75,12) -- (.75,12);
					
					\draw[nicegreen] (-.75,12) -- (0,9) -- (.75,12);
					
					\node[nicered] at (4,14.5) {$A_\alpha$};
					\node[nicegreen] at (0,13.3) {$[s\append h_A(s)]$};
					
					\node[right] at (3,6) {$\alpha=\dom(s)$};
					\node[right] at (4.5,9) {$\dom(s\append h_A(s))$};
					\node[right] at (6,12) {$\kappa$};
					
					\node[left,niceblue] at (-.1,4) {$s$};
					\node[left,nicepink] at (.2,7.5) {$h_A(s)$};

				\end{tikzpicture}\ 
				\begin{tikzpicture}[scale=.35]
					
					\draw (-3.5,7) -- (0,0) -- (4,8);
					\draw (4.75,9.5) -- (6,12);
					\draw (-6,12) -- (-4.25,8.5);
					\draw (-2,4) edge[dashed] (2,4);
					\draw (-3,6) edge[dashed] (-.5,6);
					\draw (-4.5,9) edge[dashed] (-1.5,9);
					\draw (.5,7) edge[dashed] (3.5,7);
					\draw (.5,10) edge[dashed] (5,10);
					\draw (-6,12) edge (6,12);
					
					\draw (0,0) edge[line width=1pt, in=-110, out=70, niceblue] (0,4);
					\draw (0,4) edge[line width=1pt, in=-90, out=70, looseness=1.5, nicered] (-1.5,6);
					\draw (-1.5,6) edge[line width=1pt, in=-60, out=90, looseness=1.5, nicepink] (-2.5,9);
					\draw (0,4) edge[line width=1pt, in=-90, out=70, looseness=1.5, nicegreen] (1.5,7);
					\draw (1.5,7) edge[line width=1pt, in=-110, out=90, looseness=1.5, nicepink] (1.5,10);

					\node[right] at (2,4) {$\dom(s)$};
					\node[right] at (3.5,7) {$\dom(s\append \sigma_{\gamma})$};
					\node[right] at (5,10) {$\dom(s\append \sigma_{\gamma}\append h_A(s\append\sigma_{\gamma}))$};
					\draw (10,7) edge[->,in=-60, out=40,nicepink] (10.5,9.25);
					\node[right,nicepink] at (10.5,8) {$+f^A_s(\gamma)$};
					\node[right] at (6,12) {$\kappa$};
					
					\node[right,niceblue] at (-.1,3) {$s$};
					\node[left,nicered] at (-.5,4.75) {$\sigma_{\gamma'}$};
					\node[right,nicegreen] at (.55,5) {$\sigma_{\gamma}$};
					\node[right,nicepink] at (1.25,8.75) {$h_A(s\append \sigma_{\gamma})$};
					
					\node[left,nicepink] at (-1.75,7.75) {$h_A(s\append \sigma_{\gamma'})$};

				\end{tikzpicture}\caption{Schematic overview of the definitions of $h_A$ and $f^A_s$ in the proof of \cref{dom < cof bd top}\,\textit{(1)}.}\label{figure}
			\end{figure}
			
			(II)\quad Let $g\in\nuka$ be such that $g\nleq_\rm{all} f$ for all $f\in\Cup_{A\in\cal F}D_A$. We first build sets $T_\alpha$ by induction on $\alpha\in\kappa$, which will serve as layers to a tree $T$ whose branches form the nowhere dense set $C$. Let $T_0=\ac\emp$ and given $T_\alpha$, let $t\in T_{\alpha+1}$ if and only if there is $s\in T_\alpha$, $\gamma\in\nu$ and $r\in{}^{g(\gamma)}\mu$ such that $t=s\append\sigma_\gamma\append r\append\ab0$. If $\alpha$ is limit, we let $T_\alpha$ consist of all $t\in\fkamu$ for which there is a sequence $\abset{s_\xi}{ \xi<\alpha}$ with $s_\xi\in T_\xi$ for each $\xi<\alpha$ such that $t=\Cup_{\xi\in\alpha}s_\xi$. Define the tree $T\subset\fkamu$ as the downward closure of  $\Cup_{\alpha\in\kappa}T_\alpha$, and let $C$ be the set of branches through $T$. 
			
			To see that $C$ is nowhere dense, let $s\in T$, then there are $s'\supset s$ and $\alpha\in\kappa$ such that $s'\append\ab 0\in T_{\alpha+1}$. Since every $t\in T_{\alpha+1}$ is a sequence ending in $0$, we find $s\subset s'\append \ab 1\notin T$.
			
			(III)\quad Fix some $A\in\cal F$, then we will find $k\in C\setminus A$.			
			Let $k_0=\emp\in T_0$. Given $k_\alpha\in T_\alpha\subset \fkamu$, let us write $k_\alpha=s$, then we can find $\gamma\in\nu$ such that $\dom(h_A(s\append \sigma_\gamma))= f_s^A(\gamma)<g(\gamma)$, thus there is $t\in T_{\alpha+1}$ such that $s\append\sigma_\gamma\append h_A(s\append\sigma_\gamma)\subset t$. Now note that $A_{\dom(s\append\sigma_\gamma)}\cap [s\append\sigma_\gamma\append h_A(s\append\sigma_\gamma)]=\emp$. Since $\alpha\leq\dom(s\append\sigma_\gamma)$, we also see $A_\alpha\cap [t]=\emp$. We define $k_{\alpha+1}=t$. Finally, if $\alpha$ is limit we let $k_{\alpha}=\Cup_{\xi\in\alpha}k_\xi$, and we define $k=\Cup_{\alpha\in\kappa}k_\alpha\in C$. It follows that $k\notin A_\alpha$ for all $\alpha$, so $k\notin A$.

			\item We work in the space $(\mumu,\rm{bd})$, where $\fmumu$ generates the base and $\nu=\mu^{<\mu}$. The main difference with \textit{(1)} is that functions in our space have domain $\mu$ instead of $\kappa$. We therefore use a strictly increasing continuous sequence $\abset{\mu_\alpha}{\alpha\in\kappa}$ cofinal in $\mu$ to replace $\abset{\alpha}{\alpha\in\kappa}$ at some parts in the proof. Let us write $I_\alpha=[\mu_\alpha,\mu_{\alpha+1})$ for each $\alpha\in\kappa$ and assume that $\mu_0=0$. Below we show how to translate \textit{(1)} (I--III) to the context of $(\mumu,\rm{bd})$.
			
			(I)\quad 
			The definition of $h_A\colon \fmumu\to\fmumu$ is altered so that for any $s\in{}^{\xi}\mu$ with $\xi\in I_\alpha$ we have $A_\alpha\cap\sqb{s\append h_A(s)}=\emp$. Define $f^A_s\in\nuka$ to map $\gamma\mapsto\sup\set{\alpha\in\kappa}{\mu_\alpha\leq\dom(h_A(s\append\sigma_\gamma))}$. 
			
			(II)\quad The tree $T$ that is defined from $g$ is defined exactly the same in the base case and the limit case, but in the successor stages we let $t\in T_{\alpha+1}$ if and only if there are $s\in T_\alpha$, $\gamma\in \nu$, $r\in{}^{\mu_{g(\gamma)}}\mu$ and $z\in{}^{\mu_\alpha}\mu$ such that $z\colon \xi\mapsto 0$ for all $\xi\in\mu_\alpha$ and $t=s\append\sigma_\gamma\append r\append z$. The purpose of $z$ is two-fold: by appending zeroes, we make sure that $C$ is nowhere dense, and by letting $\dom(z)=\mu_\alpha$, we make sure that the length of the sequences in $T_\alpha$ approaches $\mu$ as $\alpha$ approaches $\kappa$. The set $C\subset\mumu$ will be defined as the set of branches through $T$.
			
			(III)\quad In the construction of $k\in C\setminus A$, suppose that $k_\alpha\in T_\alpha\subset\fmumu$ and let us write $k_\alpha=s$. Then we can find $\gamma\in\nu$ such that $f_s^A(\gamma)< g(\gamma)$, and thus $\dom(h_A(s\append \sigma_\gamma))<\mu_{f_s^A(\gamma)+1}\leq\mu_{g(\gamma)}$, which proves that there is $t\in T_{\alpha+1}$ with $s\append\sigma_\gamma\append h_A(s\append\sigma_\gamma)\subset t$.\qedhere
		\end{enumerate}
	\end{proof}
	
	Before generalising the above proof to the spaces with the ${<}\mu$-box topology, we introduce two concepts. 
	Firstly, since the elements of $ \rm{Fn}_\mu(\mu,\mu)$ do not necessarily have ordinals as domain, we need to define \emph{concatenation} for two functions $s,t\in \rm{Fn}_\mu(\mu,\lambda)$, which we denote  by $s\altappend t$. Given $s,t\in \rm{Fn}_\mu(\mu,\lambda)$, we enumerate $\mu\setminus\dom(s)$ increasingly as $\abset{\beta_\xi}{\xi\in\mu}$, then we define $s\altappend t\in \rm{Fn}_\mu(\mu,\lambda)$ to be the function with domain $\dom(s)\cup \set{\beta_\xi}{\xi\in\dom(t)}$ sending $\xi\mapsto s(\xi)$ if $\xi\in\dom(s)$ and $\beta_\xi\mapsto t(\xi)$ if $\xi\in\dom(t)$.
	
	Secondly, let $\nu$ be a cardinal, then we define $\fr q_\mu(\nu)$ as the least size of a set $Q\subset{}^\nu(\sqb\mu^{<\mu})$ such that for any $g\colon \nu\to\sqb\mu^{<\mu}$ there exists $f\in Q$ such that for all $\gamma\in\nu$ we have $f(\gamma)\nsubset g(\gamma)$. 
	
	\begin{prp}
		$\nu^+\leq\fr q_\mu(\nu)$.\qed
	\end{prp}

	\begin{thm}\label{dom < cof mu top}
		\leavevmode
		\begin{enumerate}
			\item
			$\fr q_\mu(\nu)\leq \cof(\Mmumumu)$ for $\nu=\mu^{<\mu}$.
			\item
			$\fr q_\mu(\nu)\leq \cof(\Mmucsmu)$ for $\nu=\kappa^{<\mu}$.
		\end{enumerate}
	\end{thm}
	\begin{proof}
		We will follow the structure and notation of the proof of \cref{dom < cof bd top}, with the following changes: in our proof we have $\card{\cal F}<\fr q_\mu(\nu)$ and $D_A\subset{}^\nu(\sqb\mu^{<\mu})$ for each $A\in\cal F$, thus there is $g\colon \nu\to\sqb\mu^{<\mu}$ such that for each $f\in\Cup_{A\in\cal F}D_A$ there is a $\gamma\in\nu$ with $f(\gamma)\subset g(\gamma)$. Apart from this change, the structure is the same, that is, we show parts (I)--(III).
		
		\textit{(1)} 
		We work in the space $(\mumu,\mu)$, where $\rm{Fn}_\mu(\mu,\mu)$ generates the base and $\nu=\mu^{<\mu}$.  
		
		(I)\quad Define $h_A\colon \rm{Fn}_\mu(\mu,\mu)\to \rm{Fn}_\mu(\mu,\mu)$ so that for any $s\in \rm{Fn}_\mu(\mu,\mu)$ with $\card{s}\in I_\alpha$ we have $A_\alpha\cap\sqb{s\altappend h_A(s)}=\emp$, and we define $f^A_s\in{}^\nu(\sqb\mu^{<\mu})$ for $s\in \rm{Fn}_\mu(\mu,\mu)$ to map $\gamma\mapsto\dom(h_A(s\altappend \sigma_\gamma))$. Finally, note that $\card{\rm{Fn}_\mu(\mu,\mu)}=\nu$ and thus we define $D_A=\set{f_s^A}{s\in \rm{Fn}_\mu(\mu,\mu)}$ and see that $\card{D_A}\leq\nu<\fr q_\mu(\nu)$.
		
		(II)\quad We cannot build a tree, since the elements of $ \rm{Fn}_\mu(\mu,\mu)$ do not have ordinals as domain, but we may still describe $C$ as the set of branches through a \emph{tree-like} structure that we will define recursively from $g$. Let $T_0=\ac\emp$. Given $T_\alpha$, let $t\in T_{\alpha+1}$ if and only if there is $s\in T_\alpha$, $\gamma\in\nu$, $r\in{}^{g(\gamma)}\mu$ and $z\in{}^{\mu_\alpha}\mu$ such that $z\colon \xi\mapsto 0$ for all $\xi<\mu_\alpha$ and $t=s\altappend \sigma_\gamma\altappend r\altappend z$. Once again the purpose of $z$ is two-fold: to ensure that $C$ is nowhere dense and to ensure that each $f\in C$ has domain $\mu$. If $\alpha\leq\kappa$ is limit, let $t\in T_\alpha$ if and only if there are $\abset{s_\xi}{\xi\in\alpha}$ with $s_\xi\in T_\xi$ such that  $t=\Cup_{\xi\in\alpha}s_\xi$. We let $C=T_\kappa$ and claim that $C$ is nowhere dense. 
		
		Let $s\in \rm{Fn}_\mu(\mu,\mu)$ be arbitrary and suppose that $f\in [s]\cap C$. Notice that we can recursively extract sequences $\abset{\gamma_\alpha}{\alpha\in\kappa}$ and $\abset{r_\alpha}{\alpha\in\kappa}$ such that if $s_\alpha\in T_\alpha$ has $s_\alpha\subset f$, then $r_\alpha\in{}^{g(\gamma_\alpha)}\mu$ and $s_\alpha\altappend\sigma_{\gamma_\alpha}\altappend r_\alpha\subset f$. Choose $\alpha$ large enough such that $\card s<\mu_\alpha$ and let $s'=s\cup (s_\alpha\altappend\sigma_{\gamma_\alpha}\altappend r_\alpha)$, then $s'\in \rm{Fn}_\mu(\mu,\mu)$. Let $\abset{\beta_\xi}{\xi\in\mu}$ enumerate $\mu\setminus\dom(s_\alpha\altappend\sigma_{\gamma_\alpha}\altappend r_\alpha)$, then there is some $\xi\in\mu_\alpha$ such that $\beta_{\xi}\notin\dom(s)$. Now let $v\colon \ac \xi\to\ac 1$ and define $s''=s\cup (s_\alpha\altappend\sigma_{\gamma_\alpha}\altappend r_\alpha\altappend v)$, then $s''\in \rm{Fn}_\mu(\mu,\mu)$ and $s''$ is incompatible with every $t\in T_{\alpha+1}$, hence $[s'']\cap C=\emp$, thus $C$ is nowhere dense.
		
		(III)\quad This step is exactly as in \cref{dom < cof bd top}. Note that in our case for each $s=k_\alpha$ there is $\gamma\in\nu$ such that $f^A_s(\gamma)=\dom(h_A(s\altappend \sigma_\gamma))\subset g(\gamma)$, which are elements of $[\mu]^{<\mu}$. By the construction of $T_{\alpha+1}$ there is then $t\in T_{\alpha+1}$ such that $s\altappend\sigma_\gamma\altappend h_A(s\altappend\sigma_\gamma{})\subset t$.
		
		\textit{(2)} This is exactly as in \textit{(1)} above. In this case we work in the space $(\muka,\mu)$, where $\rm{Fn}_\mu(\mu,\kappa)$ generates the base and $\nu=\kappa^{<\mu}$.\qedhere
		
	\end{proof}

	\section*{Acknowledgements}
	
	We would like to thank Hiroshi Sakai for valuable discussions and for filling a gap in \cref{thm:weakly compact homeomorphism} and Nick Chapman for finding some errors in an earlier version.
	The second author also wishes to thank Hiroshi Sakai, J\"org Brendle and Diego Mej\'ia for their hospitality and support during his visits to Japan in 2024 and 2025, and Jakob Kellner and Martin Goldstern for their support in Vienna.

	\printbibliography

@article{AgostiniMottoRosSchlicht23,
	author = {Agostini, Claudio and Motto Ros, Luca and Schlicht, Philipp},
	title = {Generalized Polish spaces at regular uncountable cardinals},
	journal = {J.\ Lond.\ Math.\ Soc.},
	volume = {108},
	number = {5},
	year = {2023},
	pages = {1886–1929},
	publisher = {Wiley},
	url = {http://dx.doi.org/10.1112/jlms.12797},
	doi = {10.1112/jlms.12797},
}

@article{AndrettaMottoRos22,
	author = {Andretta, Alessandro and Motto Ros, Luca},
	title = {Souslin quasi-orders and bi-embeddability of uncountable structures},
	journal = {Mem.\ Am.\ Math.\ Soc.},
	volume = {277},
	number = {1365},
	year = {2022},
	publisher = {American Mathematical Society (AMS)},
	doi = {10.1090/memo/1365},
}

@book{BartoszynskiJudah95,
	author = {Bartoszy{\'n}ski, Tomek and Judah, Haim},
	title = {Set {T}heory: {O}n the {S}tructure of the {R}eal {L}ine},
	year = {1995},
	publisher = {A.K. Peters, Wellesley, MA},
}

@incollection{Blass10,
	author = {Blass, Andreas},
	title = {Combinatorial {C}ardinal {C}haracteristics of the {C}ontinuum},
	editor = {M. Foreman and A. Kanamori},
	booktitle = {Handbook of {S}et {T}heory {V}ol. 1},
	year = {2010},
	pages = {395--489},
	publisher = {Springer, {D}ordrecht},
	doi = {10.1007/978-1-4020-5764-9_7},
}

@article{Brendle22,
	author = {Brendle, J{\"o}rg},
	title = {The higher {C}icho{\'n} diagram in the degenerate case},
	journal = {Tsukuba J.\ Math.},
	volume = {42},
	number = {2},
	year = {2022},
	pages = {255-269},
	publisher = {University of Tsukuba},
	doi = {10.21099/tkbjm/20224602255},
}

@article{BrendleBrookeTaylorFriedmanMontoya18,
	author = {Brendle, J{\"o}rg and Brooke-Taylor, Andrew and Friedman, Sy-David and Montoya, Diana Carolina},
	title = {Cicho{\'n}’s diagram for uncountable cardinals},
	journal = {Isr.\ J.\ Math.},
	volume = {225},
	number = {2},
	year = {2018},
	pages = {959--1010},
	publisher = {Springer},
	doi = {10.1007/s11856-018-1688-y},
}

@article{CardonaMejia25,
	author = {Cardona, Miguel A. and Mej\'ia, Diego A.},
	title = {More about the cofinality and the covering of the ideal of strong measure zero sets},
	journal = {Ann.\ Pure Appl.\ Log.},
	volume = 176,
	number = 4,
	year = 2025,
	eid = 103537,
	doi = {10.1016/j.apal.2024.103537},
}

@book{ComfortNegrepontis74,
	author = {Comfort, William Wistar and Negrepontis, Stylianos},
	title = {The Theory of Ultrafilters},
	volume = {211},
	year = {1974},
	publisher = {Springer},
	doi = {10.1007/978-3-642-65780-1_10},
}

@article{CummingsShelah95,
	author = {Cummings, James and Shelah, Saharon},
	title = {Cardinal invariants above the continuum},
	journal = {Ann.\ Pure Appl.\ Log.},
	volume = {75},
	issue = {3},
	year = {1995},
	pages = {251-268},
	doi = {10.1016/0168-0072(95)00003-Y},
}

@article{DimonteIannellaLucke25,
	author = {Dimonte, V. and Iannella, M. and L\"{u}cke, P.},
	title = {Descriptive properties of I2-embeddings},
	journal = {J.\ Symb.\ Log.},
	year = {2025},
	pages = {1–26},
	publisher = {Cambridge University Press (CUP)},
	url = {http://dx.doi.org/10.1017/jsl.2024.75},
	doi = {10.1017/jsl.2024.75},
}

@misc{DimonteMottoRos25,
	author = {Dimonte, Vincenzo and Motto Ros, Luca},
	title = {Generalized descriptive set theory at singular cardinals of countable cofinality},
	year = {2025},
	eprint = {2511.16188},
	archiveprefix = {arXiv},
}

@misc{DimontePovedaThei24,
	author = {Dimonte, Vincenzo and Poveda, Alejandro and Thei, Sebastiano},
	title = {The {Baire} and perfect set properties at singulars cardinals},
	year = {2024},
	eprint = {2408.05973},
	archivePrefix = {arxiv},
}

@book{FriedmanHyttinenKulikov14,
	author = {Friedman, Sy-David and Hyttinen, Tapani and Kulikov, Vadim},
	title = {Generalized descriptive set theory and classification theory},
	volume = {1081},
	year = {2014},
	publisher = {Providence, RI: American Mathematical Society (AMS)},
	fseries = {Mem.\ Am.\ Math.\ Soc.},
	doi = {10.1090/memo/1081},
}

@article{GartiShelah20,
	author = {Garti, Shimon and Shelah, Saharon},
	title = {Remarks on generalized ultrafilter, dominating and reaping numbers},
	journal = {Fundam.\ Math.},
	volume = {250},
	number = {2},
	year = {2020},
	pages = {101–115},
	publisher = {Institute of Mathematics, Polish Academy of Sciences},
	url = {http://dx.doi.org/10.4064/fm595-9-2019},
	doi = {10.4064/fm595-9-2019},
}

@thesis{Hayashi23,
	author = {Hayashi, Yusuke},
	title = {Cardinal characteristics at singular cardinals},
	year = {2023},
	type = {Master's Thesis},
}

@article{Hayashi26,
	author = {Hayashi, Yusuke},
	title = {Dominating numbers at singular cardinals},
	journal = {Arch.\ Math.\ Log.},
	volume = {65},
	number = {4},
	year = {2026},
	pages = {531--544},
	publisher = {Springer},
	doi = {10.1007/s00153-026-01009-3},
}

@article{HungNegrepontis73,
	author = {Hung, Henry H and Negrepontis, Stelios},
	title = {Spaces homeomorphic to $(2^\alpha)_\alpha$},
	journal = {Bull.\ Am.\ Math.\ Soc.},
	volume = {79},
	number = {1},
	year = {1973},
	pages = {143--146},
	doi = {10.1090/S0002-9904-1973-13132-5},
}

@book{Jech03,
	author = {Jech, Thomas},
	title = {Set {T}heory: {T}hird {M}illennium {E}dition},
	year = {2003},
	publisher = {Springer Monographs in Mathematics},
}

@book{Kunen11,
	author = {Kunen, Kenneth},
	title = {Set Theory},
	volume = {34},
	year = {2011},
	publisher = {College Publications},
}

@book{Kunen80,
	author = {Kunen, Kenneth},
	title = {Set Theory: An Introduction to Independence Proofs},
	volume = {102},
	year = {1980},
	publisher = {North-Holland},
}

@article{LambieHanson23,
	author = {Lambie-Hanson, Chris},
	title = {A {Galvin}--{Hajnal} theorem for generalized cardinal characteristics},
	journal = {Eur. J. Math.},
	volume = {9},
	number = {1},
	year = {2023},
	pages = {21},
	note = {Id/No 12},
	doi = {10.1007/s40879-023-00610-7},
}

@article{Landver92,
	author = {Landver, Avner},
	title = {Baire numbers, uncountable {C}ohen sets and perfect-set forcing},
	journal = {J.\ Symb.\ Log.},
	volume = {57},
	number = {3},
	year = {1992},
	pages = {1086--1107},
	publisher = {Cambridge University Press},
	doi = {10.2307/2275450},
}

@article{Landver93,
	author = {Landver, Avner},
	title = {Finite combinations of {Baire} numbers},
	journal = {Isr.\ J.\ Math.},
	volume = {81},
	number = {3},
	year = {1993},
	pages = {289--296},
	doi = {10.1007/BF02764833},
}

@article{Medini11,
	author = {Medini, Andrea},
	title = {Products and h-homogeneity},
	journal = {Topol.\ Its Appl.},
	volume = {158},
	number = {18},
	year = {2011},
	pages = {2520–2527},
	publisher = {Elsevier BV},
	url = {http://dx.doi.org/10.1016/j.topol.2011.08.011},
	doi = {10.1016/j.topol.2011.08.011},
	doi = {10.1016/j.topol.2011.08.011},
}

@article{MeklerVaananen93,
	author = {Mekler, Alan and V{\"a}{\"a}n{\"a}nen, Jouko},
	title = {Trees and $\Pi^1_1$-subsets of ${}^{\omega_1}\omega_1$},
	journal = {J.\ Symb.\ Log.},
	volume = {58},
	number = {3},
	year = {1993},
	pages = {1052--1070},
	publisher = {Cambridge University Press},
}

@article{Miller82,
	author = {Miller, Arnold W},
	title = {The Baire category theorem and cardinals of countable cofinality},
	journal = {J.\ Symb.\ Log.},
	volume = {47},
	number = {2},
	year = {1982},
	pages = {275--288},
	publisher = {Cambridge University Press},
	doi = {10.2307/2273142},
}

@article{Shelah19,
	author = {Shelah, S.},
	title = {On $\mathfrak{d}_\mu$ for $\mu$ singular},
	journal = {Acta Math.\ Hung.},
	volume = {161},
	number = {1},
	year = {2019},
	pages = {245–256},
	publisher = {Springer Science and Business Media LLC},
	doi = {10.1007/s10474-019-00999-2},
}

@book{Shelah94,
	author = {Shelah, Saharon},
	title = {Cardinal Arithmetic},
	year = {1994},
	publisher = {Oxford University Press},
	doi = {10.1093/oso/9780198537854.001.0001},
}

@article{Sikorski50,
	author = {Sikorski, Roman},
	title = {Remarks on some topological spaces of high power},
	journal = {Fundam.\ Math.},
	volume = {37},
	year = {1950},
	pages = {125--136},
	doi = {10.4064/fm-37-1-125-136},
	url = {https://eudml.org/doc/213209},
}

@article{Stone62,
	author = {Stone, A. H.},
	title = {Non-separable Borel sets},
	journal = {Rozprawy Mat.},
	volume = {28},
	year = {1962},
	url = {https://pldml.icm.edu.pl/pldml/element/bwmeta1.element.zamlynska-cbe30d40-7f84-46e7-9066-30f1f502ae47},
}

@article{Terada93,
	author = {Terada, T.},
	title = {Spaces whose all nonempty clopen subspaces are homeomorphic},
	journal = {Yokohama Math. J.},
	volume = {40},
	year = {1993},
	pages = {87--93},
}
	
\end{document}